\documentclass[a4paper,onecolumn,superscriptaddress,10pt,accepted=2025-01-17,issue=4,volume=5, shorttitle=papers/compositionality-5-4]{compositionalityarticle}
\pdfoutput=1 
\usepackage[utf8]{inputenc}
\usepackage[english]{babel}
\usepackage[T1]{fontenc}
\usepackage{amsmath,amsthm}
\usepackage{hyperref}
\usepackage{lipsum}
\usepackage[all]{xy}
\usepackage{euscript,enumitem}
\usepackage{mathtools,amssymb,amsfonts,mathrsfs,graphicx,amsmath,bbold}
\usepackage{yfonts,color}

\newtheorem{theorem}{Theorem}
\newtheorem{corollary}[theorem]{Corollary}

\newtheorem{lemma}[theorem]{Lemma}
\newtheorem{proposition}[theorem]{Proposition}

\theoremstyle{definition}
\newtheorem{example}[theorem]{Example}
\newtheorem{nonexample}[theorem]{Non-Example}
\newtheorem{remark}[theorem]{Remark}
\newtheorem{definition}[theorem]{Definition}

\def\bha{{\boldsymbol \eta}}
\def\bmu{{\boldsymbol \mu}}
\def\semi{{\underline{\Fin}}}
\def\bzetaodd{\raisebox{.7em}{\rotatebox{180}{$\bzeta$}}}
\def\Rada#1#2#3{#1_{#2},\dots,#1_{#3}}
\def\copcat{{\tt COpCat}}
\def\COPCAT{{\mathbf {COpCat}}}
\def\bigstar{\star}
\def\M{{\EuScript M}}
\def\gr{{\tt GR}}
\def\Im{{\it Im}}
\def\prGR{{\bf prGR}}
\def\Leg{{\rm Leg}}
\def\Flag{{\rm Flg}}
\def\Vert{{\rm Ver}}
\def\Edg{{\rm Edg}}
\def\bA{{\mathbf A}}
\def\Vines{{\tt Vin}}
\def\ttV{{\tt V}}
\def\rada#1#2{{#1,\ldots,#2}}

\def\dash{{\hbox{\hskip .15em -\hskip .1em}}}

\def\unit{{{\mathbb I}_{\,\ttV}}}
\def\GR{{\mathbf {GR}}}
\def\sou{\bs} \def\tar{\bt}
\def\bR{{\boldsymbol R}}
\def\bzeta{{\boldsymbol \zeta}}
\def\ba{{\boldsymbol a}}
\def\bb{{\boldsymbol b}}
\def\bc{{\boldsymbol c}}
\def\colim{\mathop{\rm colim}\displaylimits}
\def\bt{{\mathbf t}}
\def\bs{{\mathbf s}}
\def\ot{{\otimes}}
\def\oP{{\EuScript P}}
\def\boP{\boldsymbol {\EuScript P}}
\def\bal{\boldsymbol {\alpha}}
\def\boE{\boldsymbol {\EuScript E}}
\def\oE{{\EuScript E}}
\def\btf{\boldsymbol {\tilde f}}
\def\id{{\mathbb 1}}
\def\bU{{\boldsymbol U\!}}

\def\bF{{\boldsymbol F}}
\def\E{{\mathbf E\,}}
\def\R{{\rm R\,}}
\def\sFin{{\tt Fin}}
\def\inv#1#2{#1^{-1}(#2)}
\def\rep#1#2#3#4{{(#1 \!\! \stackrel{#2}\longrightarrow \!
    \underline#3\, \raisebox{.6em}{\rotatebox{180}{$\rightsquigarrow$}} \ #4})}  
\def\lsq{\xymatrix@C=1em@1{&\ar@{~>}[l]}}
\def\bS{{\boldsymbol S}}
\def\bQ{{\boldsymbol Q}}
\def\bT{{\boldsymbol T}}
\def\bg{{\boldsymbol g}}
\def\bR{{\boldsymbol R}}
\def\boldf{{\boldsymbol f}\!}
\def\bg{{\boldsymbol g}}
\def\bh{{\boldsymbol h}}

\def\Gr{{\tt Gr}}

\def\crd{\hbox{$|$-$|$}}
\def\ttE{{\tt E}}
\def\ttB{{\tt B}}
\def\ttO{{\tt O}}
\def\bO{{\boldsymbol {\mathbf O}}}
\def\Fin{\hbox{\bf {Fin}}}

\def\tilde{\widetilde}

\newcommand{\authorsforheader}{M.\ Markl}
\newcommand{\compositionalityIssue}{5}
\newcommand{\compositionalityVolume}{7}
\newcommand{\paperdoinumber}{10.46298/compositionality-\compositionalityVolume-\compositionalityIssue}
\newcommand{\paperdoi}{https://doi.org/\paperdoinumber}
\newcommand{\receiveddate}{16 March 2025}
\newcommand{\accepteddate}{26 September 2025}

\title[Operadic categories with cardinalities in finite sets]{Cloven
  operadic categories:
An approach to operadic categories with cardinalities in finite unordered sets}
\author{Martin Markl}
\orcid{0000-0003-4611-8226}

\affiliation{The Czech Academy of Sciences, Institute of Mathematics, {\v Z}itn{\'a} 25,
         115 67 Prague 1, The Czech Republic}

\thanks{Supported by  RVO: 67985840.}

\begin{document}

\begin{abstract}
We introduce and study operadic categories with cardinalities in
finite sets and establish conditions under which their associated
theories of operads and algebras are equivalent to the standard framework
developed in~\cite{duodel}. Our approach is particularly natural in
applications to the operadic category of graphs and the related category of
modular operads and their clones.
\end{abstract}

\maketitle

\tableofcontents

\section*{Introduction}

Recall that the traditional operads~\cite{Artamonov,may:1972} in a
symmetric monoidal category $\ttV$, which is typically the category of
sets or graded vector spaces,
are collections $\oP=\{\oP(n)\, | \, n\geq 1\}$ of objects of~$\ttV$
such that
\begin{subequations}
\setlength{\leftmargini}{3em}
\begin{itemize}[topsep=0pt, partopsep=0pt, itemsep=0pt]
\item[(i)]
each $\oP(n)$ is a module over the symmetric group $\Sigma_n$, and
\item[(ii)]
there are  composition laws
\begin{equation}
\label{po 6ti dnech}
\mu_\rho :
\oP(k) \ot \oP(n_1) \ot \cdots \ot \oP(n_k) \longrightarrow
\oP(m),\
m = n_1 + \cdots +
n_k,\ k \geq 1,\ \Rada n1k \geq 1,
\end{equation} 
\end{itemize}
satisfying appropriate
associativity and equivariance axioms; we assumed for simplicity 
that there is no $\oP(0)$. The
meaning of the index $\rho$ at $\mu_\rho$ is explained below. We call
the above form of definition the {\/\em skeletal\/} definition of
(classical) operads, the natural numbers featured in~(\ref{po 6ti
  dnech}) are the {\/\em arities\/} of the corresponding components.
An equivalent, {\/\em non-skeletal\/} definition with 
arities in finite sets was given
in~\cite[Section~II.1.7]{markl-shnider-stasheff:book}. Operads in this setup
appear as collections \hbox{$\boP = \{\, \boP(X)\, | \, \hbox {$X$ a finite
  set\ }\}$}, with
two kinds of data:
\setlength{\leftmargini}{3em}
\begin{itemize}[topsep=0pt, partopsep=0pt, itemsep=0pt]
\item [(i)]
a natural morphism $\varphi_* : \boP(X') \to \boP(X'')$ specified for each
isomorphism $\varphi: X' \to X''$ of finite sets, and
\item[(ii)]
composition laws
\begin{equation}
\label{zablokovana karta}
\mu_g:\boP(Y) \ot \bigotimes_{y \in Y} \boP(\inv gy) \longrightarrow \boP(X),
\end{equation}
given for each epimorphism $g : X \to Y$ of finite sets,
\end{itemize}
satisfying again appropriate axioms. In~\eqref{zablokovana karta},
$\inv gy$ is the set-theoretic preimage of $y \in Y$, and
the display involves
the ``unordered tensor
product''~\cite[Definition~I.1.58]{markl-shnider-stasheff:book} 
of the preimages indexed by the
{\em unordered\/} set $Y$.  
\end{subequations}

Notice that the operations in~(\ref{po 6ti dnech}) are 
parametrized by maps $\rho :\{\rada 1m\} \to \{\rada 1k\}$ so that
$n_i$ appears as the cardinality of the preimage $\inv
\rho i$, $i \in \{\rada 1k\}$. We emphasize that $\rho$
can be an {\/\em arbitrary\/}, not necessarily order-preserving
map. Therefore both~(\ref{po 6ti dnech}) and~(\ref{zablokovana karta}) are  
parametrized by maps of finite sets. The sets in~(\ref{po 6ti dnech})
belong to the small category of finite ordinals.
Their order gives a preferred order of the
preimages  $\inv \rho i$, which circumvents the use of the unordered
tensor product. On the other hand,~(\ref{zablokovana karta}) refers to the
big category of all finite sets, but the formula is ``canonical,'' 
since it does not refer to any particular choice of the
skeletal subcategory of finite sets, realized by finite ordinals
in~\eqref{po 6ti dnech}. 

Let us briefly mention non-$\Sigma$ (non-symmetric) operads. In the traditional
presentation, such an operad $\underline\oP$ 
is a collection of objects of $\ttV$, no group actions assumed, with
$\rho$ in~(\ref{po 6ti dnech}) order-preserving. In the
non-skeletal setup, it is a collection $\underline\boP$
indexed by finite ordered sets, with composition
laws~(\ref{zablokovana karta}) parametrized by {\em order-preserving\/}
maps. But to keep both definitions equivalent, we still need
natural actions  $\varphi_*: \underline\boP(X') \to \underline\boP(X'')$ of
order-preserving maps which do not appear in the skeletal version.

The preference for the skeletal version for classical operads is
quite understandable. However, the picture changes for cyclic operads and their
clones, such as modular operads. 
Here the skeletal version requires the choice of a skeletal
subcategory of the category of finite sets with the cyclic group
action, and a ``nice'' linear
formula as in~(\ref{po 6ti dnech}) moreover requires non-canonical choices of
linear orders of cyclically ordered sets. The resulting formulas are
so clumsy as to be practically useless, cf.~\cite[Definition~6.23]{DMJ}.  
On the other hand, the non-skeletal version, written out 
in~\hbox{\cite[Definition~A.1]{kodu}} or~\cite[Definition~6.16]{DMJ}, is
sensible and intuitively clear.

So both approaches have their merits, depending on the context. The
applications we had had in mind, together with the motivating example of
Batanin's $k$-trees, lead us in~\cite{duodel} to the definition of operadic
categories over finite ordinals. Our aim
is to allow {\/\em
  all\/} finite sets and to determine when both approaches lead to equivalent notions
of operads and their algebras. A~typical example for the later approach
is the category of graphs for which, in contrast to the skeletal version, no
orders of vertices and (half)edges are required. We believe that
extending the concept of operadic categories to allow the 
standard, ``na\"\i ve'' approach to graphs  will make the theory more
amenable for applications. 

An anonymous referee suggested the following example of a situation
in which one is naturally led to consider operads with arities in
arbitrary finite sets. Let $\oP =\{\oP(n) \mid n\geq 1\}$ be the
standard classical operad as in~(\ref{po 6ti dnech}), with the
structure expressed in the equivalent form of
$\circ_i$-operations
\begin{subequations}
\begin{equation}
\label{Pisi uz v Bonnu}
\circ_i : \oP(m) \otimes \oP(n) \longrightarrow \oP(m+n-1), \ m,n
\geq 1, 1 \leq i \leq m.
\end{equation}
The frequently used bar construction $B(\oP)$ of $\oP$ 
is formed by rooted trees with
decorated vertices. The decoration is such that a vertex $v$ with 
$n$ incoming edges is decorated by the component $\oP(n)$ of
$\oP$. The differential is given by the sum of edge
contractions. However, since the incoming edges of $v$ are unordered,
it is not possible to determine which structure map~(\ref{Pisi uz v Bonnu}) of
$\oP$ corresponds to a~given edge contraction. A standard
way to overcome this difficulty is to allow the operad to have
arities in arbitrary finite sets, and the structure operations of the form
\begin{equation}
\label{Jarka mi vcera zalila kyticky.}
\circ_e : \oP(X) \otimes \oP(Y) \longrightarrow \oP((X \setminus
\{e\}) \cup Y), \ e \in X, \ X\cap Y = \emptyset.
\end{equation}
\end{subequations}
Then the vertex $v$ is
decorated by the component $\oP(X)$ of this extended operad,
with $X$ the (finite) set of incoming edges of $v$. The
structure operation corresponding to the contraction of an incoming edge
$e$ of $v$ is precisely $\circ_e$ in~(\ref{Jarka mi vcera zalila kyticky.}). The
applicability of this procedure relies on the equivalence
between the
`standard' definition of operads and operads with arities in
finite sets, cf.~\cite[Section~II.1.7]{markl-shnider-stasheff:book} again. 
Such an extension is standard when working with
operads and their generalization. The goal of the present article is
to lift this procedure to the level of operadic categories governing
the given type of operads.

\noindent 
{\bf Terminology and notation.} We will call operadic categories over
finite ordinals introduced in~\cite{duodel} {\/\em thin\/}, those
over arbitrary finite sets treated
in this article  {\/\em thick\/}. In other words, thin stands for 
skeletal and thick stands for non-skeletal. 
Whenever it makes sense to do so, the symbols 
for objects related to thick operadic categories are typed in {\bf bold}.
For instance, the cardinality functor of a thin operadic
category $\ttO$ will be written as $\crd : \ttO \to \sFin$, its thick version
as \hbox{$\crd : \bO \to \Fin$}. 

Operadic
categories and their operads in~\cite{duodel} were unital by
definition. However, in our subsequent work~\cite{blob} we realized
that for some applications is is better to consider unitality as an
additional, independent, property. We follow this convention 
in the present article.

\noindent 
{\bf Results and novelties.}  Theorem~\ref{Dnes jedu do
  Kolina a zitra letim do Dvora.} establishes, and its proof
explicitly describes, a natural equivalence between the
categories of thin and thick operadic categories, respectively, such that the
categories of operads and their algebras of the corresponding operadic
categories are naturally equivalent. For such equivalences to exist, we must
restrict our focus to operadic categories which are {\em cloven\/}.
This property, introduced in Definitions~\ref{8 dni na chalupe!}
and~\ref{Zase je to nejiste.}, requires the existence of a {\em cleavage\/} -- a
functorial choice of certain lifts.  The
categories of graphs analyzed in detail in the last section provide
convincing and easy to work with examples of thick operadic categories.

Regarding the applicability of our approach,
most operadic categories that a working
mathematician may encounter are either
ordered, meaning by definition that the cardinality functor
factorizes through the category of finite ordinals
and order-preserving maps, or cloven. Ordered operadic
categories can be treated by a simpler version of our
theory. The only ``non-synthetic''
examples of operadic categories that are neither cloven nor ordered
we know of are the operadic category of Batanin’s $k$-trees
and the operadic category of vines, cf.~non-Examples~\ref{za necele
  tri hodiny} and~\ref{Dnes_na_muslicky.}. 
Our approach does not produce thick counterparts of these
categories with equivalent categories of operads and
their algebras, and we doubt that they exist.

Let us see what happens if we replace finite sets by finite {\/\em
  ordered\/} sets.
For the lack of better terminology we will call an
operadic category {\em semi-ordered\/}, if a factorization of the cardinality
functor through the category of ordered sets and their,
not-necessarily order-preserving, morphisms is given. Obviously, a~thin operadic
category is always semi-ordered. A semi-ordered operadic
category is ordered,  if the cardinality
functor factorizes through the category of ordered sets and
{\em order-preserving\/} maps. A suitably modified notion of a cleavage
makes sense for ordered and semi-ordered operadic
categories; thin operadic
categories are always cloven in this modified sense.

In~Theorem~\ref{V patek exkurze do Caslavi.} we establish an equivalence between
the category of thin operadic categories and the category of thick
semi-ordered cloven operadic categories, which further restricts to an
equivalence of the subcategories of ordered operadic categories.  This
equivalence has the expected property that the corresponding
operadic categories have equivalent categories of operads and their
algebras. We are aware that this setup is not very interesting, as
it does not achieve our main goal of avoiding the use of ordered
sets. We include it only for the sake of completeness.

\noindent
{\bf Plan of the paper.}
In Sections~\ref{Jak dlouho bude ta ``rekonstrukce'' trvat?}--\ref{MR
  byla v poradku!}, we introduce thick 
operadic categories, the concept of cleavage, the related notions of
operads and their algebras, and prove their basic properties. Section~\ref{Dnes je
  Mikulase, a Mikulasek s Lachtankem jeste spinkaji.} contains the
main results of the paper. It introduces two functors --
the extension and the restriction -- between the category of cloven thin and
the category of cloven thick operadic categories 
which together form a pair of mutually inverse equivalences. We
show that the corresponding categories of
operads and their algebras are equivalent. Section~\ref{co tri dny} is
devoted to a~detailed study of the thick version of the category of
graphs. In the Appendix we briefly sketch the (semi)ordered version of
our theory.

\noindent 
{\bf Requirements.} The standard, ``thin''
operadic categories, their operads
and algebras were introduced in~\cite{duodel}.
We will assume familiarity with  this apparatus. 
Some knowledge of the classical operads and PROPs
may ease reading, cf.~\cite{markl:handbook} or 
the monograph~\cite{markl-shnider-stasheff:book}.
The present article is a loose continuation of the
series~\cite{env,kodu,blob} of papers devoted to general properties of
operadic categories, so acquaintance with  those sources is also welcome.

\noindent {\bf Acknowledgment.} The author wishes to thank the referee
for carefully reading the paper and for a concrete suggestion that led to a
paragraph in the introduction, helping the reader to better understand
the context of the results of this article. The author is also
grateful to the Max Planck Institute for Mathematics in Bonn for its
hospitality and financial support.

\section{Thick operadic categories}
\label{Jak dlouho bude ta ``rekonstrukce'' trvat?}

Operadic categories were introduced in~\cite{duodel} and further
studied in several subsequent papers. Their salient feature is that
each object has its cardinality, which is a finite ordinal. We shall
refer to these categories as {\em thin} operadic
categories. The aim of this section is to introduce a ``thick'' 
version with cardinalities in arbitrary finite sets.  

\noindent 
{\bf Finite sets.}
In the following text, $\Fin$ will denote the category of (all) finite sets.
For a map $\phi: S \to T$ in $\Fin$, $\inv \phi t$ will be the
set-theoretic preimage of $t \in T$. By $\sFin$ we denote the
skeletal category of finite sets. Its objects are finite ordinals 
$\underline n := (1,\ldots,n)$, $n \geq 0$, with the convention that 
$\underline 0$ is the empty set. The morphisms in $\sFin$ are all, not
necessarily order-preserving, maps. For $f : \underline m \to
\underline n$ and $i \in \underline n$ we denote by $\inv fi$ the
pullback of $f$ along the map $\underline 1 \to \underline n$ which picks
up $i \in \underline n$, as in
\begin{equation}
\label{Zase tady nekdo neco hrabe!}
\xymatrix@C=1.5em@R=-.1em{\inv fi\ar[dd] \ar[rr]^{I_i} && \underline m \ar[dd]^{f}
\\
&\hbox{\raisebox{1.5em}{\huge \hskip -2.3em  $\lrcorner$}}
\\
\underline 1 \ar[rr]^{1 \ \longmapsto \ i}     &&\ \underline n. 
}
\end{equation}
The set $\inv fi$ is unique by the skeletality of \/ $\sFin$.  We will
need at some
places in Section~\ref{Dnes je Mikulase, a Mikulasek s Lachtankem
  jeste spinkaji.} to emphasize that  $\inv fi$ is a pullback,
not an inverse image, we will thus write more specifically $f_{\it
  pb}^{-1}(i)$ instead.

We are going to define operadic categories with
arbitrary finite sets as cardinalities. They are straightforward
modifications of ``thin'' operadic categories as defined
in~\cite[Section~1]{duodel} which assumed cardinalities in $\sFin$.
We will however leave out all references to chosen local terminal
objects, since the related issue of unitality has a different flavor
in the non-skeletal case, cf.~the remarks following the definition. The
same definition appeared also in~\cite{BJM}.

\begin{definition}
\label{Uz treti tyden je strasliva zima.}
\begin{subequations}
A {\/\em thick (strict, nonunital) operadic
  category\/} is a 
category $\bO$  equipped with a~``cardinality''
functor  $|\dash|:\bO\to \Fin$ with the following properties.
For every $\boldf:\bT\to \bS$
in $\bO$ and every element $s\in |\bS|$ there is given an object
$\inv \boldf s \in \bO,$ {\it the $s$th fiber\/} of~$\boldf$,
such that 
\begin{equation}
\label{PSA}
|\inv \boldf s| = |\boldf|^{-1}(s).
\end{equation}
We moreover require that
\setlength{\leftmargini}{2em}
\begin{itemize}[topsep=0pt, partopsep=0pt, itemsep=0pt]
\item[(i)] 
for any commutative diagram
  \begin{equation}
    \label{Kure}
    \xymatrix@C = +1em@R = +1em{
      \bT      \ar[rr]^\boldf \ar[dr]_\bh & & \bS \ar[dl]^\bg
      \\
      &\bR&
    }
\end{equation}
 in $\bO$ and every $r\in |\bR|$, one is given a morphism
\[
\boldf_r: \bh^{-1}(r)\to \bg^{-1}(r)
\]
in $\bO$ such that $|\boldf_r|: |\bh^{-1}(r)|\to |\bg^{-1}(r)|$ equals the map
$|\bh|^{-1}(r)\to |\bg|^{-1}(r)$ of the preimages 
induced by
\[
\xymatrix@C = +1em@R = +1em{ |\bT| \ar[rr]^{|\boldf|} \ar[dr]_{|\bh|} & & |\bS|
\ar[dl]^{|\bg|}
\\
&|\bR|& }.
\]
This assignment must moreover assemble to a functor ${\rm Fib}_r: \bO/\bR \to \bO$.
\item[(ii)] 
In the situation of~\eqref{Kure}, for any $s\in |\bS|$,  one has the equality
\begin{equation}
\label{karneval_u_retardacku}
\boldf^{-1}(s) = \boldf_{|\bg|(s)}^{-1}(s).
\end{equation} 
\item[(iii)] 
Let 
\[
\xymatrix@C = +2.5em@R = +1em{ & \bS \ar[dd]^(.3){\bg} \ar[dr]^\ba & 
\\
\bT  \ar[ur]^\boldf    \ar@{-}[r]^(.7){\bb}\ar[dr]_\bh & \ar[r] & \bQ \ar[dl]^\bc
\\
&\bR&
}
\]
be a commutative diagram in $\bO$, and let $q\in |\bQ|, r :=
|\bc|(q) \in |\bR|$. Then
\[
\xymatrix@C = +1em@R = +1em{
\bh^{-1}(r)      \ar[rr]^{\boldf_r} \ar[dr]_{\bb_r} & & g^{-1}(r) \ar[dl]^{\ba_r}
\\
&\bc^{-1}(r)&
}
\]
commutes by (i), so it induces a morphism 
$(\boldf_r)_q: \bb_r^{-1}(q)\to \ba_r^{-1}(q)$. We have by~\eqref{karneval_u_retardacku} 
\[
\ba^{-1}(q)=\ba_r^{-1}(q) \ \mbox{and} \ \bb^{-1}(q)=\bb_r^{-1}(q).
\]
We  then require that $\boldf_q = (\boldf_r)_q$.
\end{itemize}
We will also assume that the set $\pi_0(\bO)$ of connected components
is small with respect to  a sufficiently big ambient universe.
\end{subequations}
\end{definition}  

Equation~\eqref{PSA} implies that the
fibers $\inv \boldf s$,  $s \in \bS$, of a morphism $\boldf: \bT \to
\bS$ in $\bO$ 
are mutually disjoint and the cardinality $|\bT|  \in \Fin$ of $\bT$
equals their disjoint union
\begin{equation}
\label{Vcera jsem byl s Jaruskou v Brne.}
|\bT| = \bigsqcup_{s \in |\bS|} |\inv \boldf s|.
\end{equation}

Let us recall the concept of unitality for ``thin'' operadic categories
$\ttO$ introduced in~\cite[Section~1]{duodel} and further 
refined in~\cite[Section~2]{blob}.

\begin{definition}
\label{Poleti se zitra?}
Suppose that  a family 
\begin{equation}
\label{Jsem posledni den v Cuernavace.}
\big\{U_c \in \ttO \ | \ c \in \pi_0(\ttO) \big\}
\end{equation}
of local terminal objects of a thin operadic category $\ttO$ 
is specified, with $U_c$
belonging to the connected component~$c$, and such that $|U_c| =
\underline 1$ for each $c$.

Then $\ttO$ is {\em
  left unital\/} if all fibers of the identity
automorphism $\id:S \to S$ belong to the set~\eqref{Jsem posledni den
  v Cuernavace.} 
of the chosen local terminal objects, 
for each $S \in \ttO$.  The category $\ttO$ is {\em right
  unital\/} if the fiber functor ${\rm Fib}_1 : \ttO/U_c \to \ttO$ is the domain functor
for each $c \in \pi_0(\ttO)$. Finally, $\ttO$ is {\em unital\/} if it
is both left and right unital. 
\end{definition}

Definition~\ref{Svezu zitra Jarusku Sambou?} formally makes sense also
in the thick case. While the right unitality occurs quite often, a thick
operadic category $\bO$ can be left unital only if $|\bT| = \{1\}$ for
all $\bT \in \bO$. Indeed, the cardinalities of the fibers of $\id :
\bT \to \bT$ are $\{t\}$, $t \in |\bT|$, by~\eqref{PSA}, but we required that
all the chosen local terminal objects have cardinalities $\{1\}$.
It is more natural to assume only that the fibers of
identity automorphisms are {\em isomorphic} to the chosen local terminal
objects, as done in Definition~\ref{Pisu v Mercine}, Section~\ref{V
  sobotu jdu s Jarkou na Vystaviste.}. 

\begin{example}
\label{Svezu zitra Jarusku Sambou?}
The category $\Fin$ is a thick operadic category, with the cardinality
given by the identity functor $\Fin \to \Fin$. It is right unital,  
with the family~\eqref{Jsem posledni den v Cuernavace.} consisting of
a single object $\{1\}$. More sophisticated examples of thick unital
operadic categories are provided by various categories of graphs
analyzed in Section~\ref{co tri dny}.
\end{example}

\section{Cloven operadic categories}

Given a functor $F: \ttE \to \ttB$ and objects $S \in \ttE$, $X \in \ttB$, we
express that $F(S) = X$ by $S \rightsquigarrow X$. This notation will
typically be used when $F$ is a cardinality functor of an operadic category.
The definition below has been inspired by the notion of a normal splitting
cleavage of a Grothendieck fibration. However, we will require no
fibration property of $F$, and the lifts will be required only for
isomorphisms.

\begin{definition}
\label{Pojedeme na tu Sumavu?}
A {\/\em cleavage\/} for $F: \ttE \to \ttB$ is a choice, for each
isomorphism $\sigma : X \to Y$ in $\ttB$ and an object 
$S \in \ttE$ with $F(S) =
X$, of an object
$T \in \ttE$ and of a {\em lift\/} $\tilde \sigma : S \to T$ such that
$F(\tilde \sigma) = \sigma$, i.e.\ 
\begin{equation}
\label{Pujdeme k Pakousum?}
\xymatrix@C=4em{S \ar@{-->}[r]^{\tilde \sigma} \ar@{~>}[d]& T  \ar@{~>}[d]
\\
X \ar[r]^\sigma_\cong &\ Y.
}
\end{equation}    
We moreover require the functoriality of the lifts, 
that is $\tilde{\sigma''\sigma''} =
\tilde{\sigma''} \tilde{\sigma'}$ in
\[
\xymatrix@C=4em{
\ar@/^2em/@{-->}[rr]^{\tilde{\sigma''\sigma''}}
S \ar@{-->}[r]^{\tilde {\sigma'}} \ar@{~>}[d]& T \ar@{-->}[r]^{\tilde {\sigma''}}  \ar@{~>}[d]&R \ar@{~>}[d]
\\
X \ar[r]^{\sigma'}_\cong &Y  \ar[r]^{\sigma''}_\cong&Z.
}
\]  
and that  the identities 
are lifted to the identities.
We then say that the functor $F$ is {\em cloven\/}.
\end{definition}

\begin{remark}
Since the lifts are required only for isomorphism, their functoriality
implies that the chosen lift in~(\ref{Pujdeme k Pakousum?}) is uniquely
determined either by its domain, or by its target. We will call the
diagrams as in~(\ref{Pujdeme k Pakousum?}) the {\/\em lifting squares\/}.
\end{remark}

\begin{definition}
\label{8 dni na chalupe!}
A thick operadic category $\bO$ is {\/\em cloven\/} if the cardinality
functor $\crd : \bO \to \Fin$ is cloven by a cleavage
compatible with the fiber
functor in the following sense. 
Every choice of isomorphisms $\rho :X \to Y$
and $\sigma: A \to B$ in $\Fin$ together with objects $\bS,\bT \in \bO$ such
that $|\bS| = X$, $|\bT|=A$, connected by a morphism $\boldf:\bS \to
\bT$, determine the remaining objects in the~diagram  
\begin{equation}
\label{Ve ctvrtek letim do Dvora.}
\xymatrix@C=2.3em@R=1.5em{& 
\ar@{-->}[rr]^{\btf} \ar@{~>}'[d][dd] \tilde \bS&& \tilde \bT \ar@{~>}[dd]
\\
\bS \ar@{-->}[ur]^{\tilde \rho}  \ar[rr]^(.6)\boldf  \ar@{~>}[dd] && 
\bT \ar@{-->}[ur]^{\tilde
  \sigma}  
\ar@{~>}[dd]
\\
&Y \ar@{-->}'[r]^(.7){|\btf|}[rr] && B
\\
X \ar[rr]^{|\boldf|}\ar[ur]^\rho_\cong  && A \ar[ur]^\sigma_\cong
}
\end{equation}
with commuting upper and bottom squares, and whose left and right faces
are lifting squares.

Namely, $\tilde \rho: \bS \to \tilde \bS$ is the lift of $\rho$, $\tilde
\sigma : \bT \to \tilde \bT$ the lift of $\sigma$ and $\btf := \tilde
\sigma \boldf {\tilde \rho}^{-1}$.  
For each $a \in A$ and $b := \sigma(a) \in B$ we have 
the induced isomorphism $\rho_a:   |\boldf|^{-1}(a) \to 
|\tilde \boldf|^{-1}(b)$ between
the set-theoretic preimages. We require the existence of the lifting square
\begin{equation}
\label{7my den na chalupe}
\xymatrix@C=3em{
\boldf^{-1}(a) \ar@{~>}[d]   \ar@{-->}[r]^{\tilde {\rho_a}}
& 
{\btf}^{-1}(b) \ar@{~>}[d] 
\\
|\boldf|^{-1}(a)  \ar[r]^{{\rho_a}}_\cong    & {|\btf|}^{-1}(b)
}
\end{equation}
where ${\btf}^{-1}(b)$ is the fiber of $\btf$ over $b \in B$.
\end{definition}

\begin{nonexample}
\label{tyden po vysetreni}
The thick operadic category ${\mathbf 1}$
with one object ${\Large\bullet}$ of cardinality $\{1\}$ is
not cloven -- there is no lift of the isomorphism $\sigma$ in the
diagram
\[
\xymatrix@C=4em{\bullet \ar@{-->}[r]^{\tilde \sigma} \ar@{~>}[d]& ?  \ar@{~>}[d]
\\
\{1\} \ar[r]^\sigma_\cong &\ \{x\}
}
\]
if $x \not= 1$. All examples of non-cloven thick operadic categories we know
are of this or similar ``synthetic'' type.
\end{nonexample}

\begin{example}
\label{Pisu na Safari.}
The thick operadic category $\Fin$, with the cleavage given by $\tilde
\sigma : = \sigma$ for any isomorphism in $\Fin$, is cloven. 
This follows from the simple fact that the upper and bottom squares
in~\eqref{Ve ctvrtek letim do Dvora.} are the same, so the same are
also the upper and bottom horizontal maps in~\eqref{7my den na chalupe}.
Typical examples of thick cloven operadic categories are provided by
the categories of (unlabeled, non-oriented) graphs presented in Section~\ref{co tri dny}.
\end{example}

\begin{definition}
\label{Zase je to nejiste.}
A standard, ``thin'' operadic category $\ttO$ as
in~\cite[Section~1]{duodel} is {\/\em cloven\/} if the cardinality
functor $\crd : \ttO \to \sFin$ is cloven compatibly with the fiber
functor in the sense analogous to that in Definition~\ref{8
  dni na chalupe!}. Explicitly, every automorphisms (permutations) $\rho :
\underline m \to \underline m$,  $\sigma : \underline n \to \underline n$
together with objects $S,T \in \ttO$ with 
$|S| = \underline m$,  $|T|=\underline n$, that are
connected by a morphism $f:S \to
T$, determine the diagram  
\begin{equation}
\label{Zitra se budeme divat na manzele Stodolovy.}
\xymatrix@C=2.3em@R=1.5em{& 
\ar@{-->}[rr]^{\tilde f} \ar@{~>}'[d][dd] \tilde S&& \tilde T \ar@{~>}[dd]
\\
S \ar@{-->}[ur]^{\tilde \rho}  \ar[rr]^(.6)f  \ar@{~>}[dd] && 
T \ar@{-->}[ur]^{\tilde
  \sigma}  
\ar@{~>}[dd]
\\
&\underline m \ar@{-->}'[r]^(.7){|\tilde f|}[rr] && \underline n
\\
\underline m \ar[rr]^{|f|}\ar[ur]^\rho_\cong  && \underline n \ar[ur]^\sigma_\cong
}
\end{equation}
analogous to~\eqref{Ve ctvrtek letim do Dvora.}.
We have for, each $i \in \underline
n$ and $j := \sigma(i)$, 
the canonical induced isomorphism  \hbox{$\rho_i:   |f|^{-1}(i) \to 
|\tilde f|^{-1}(j)$} between the pullbacks and
require the existence of the lifting diagram
\begin{equation}
\label{Zitra pojedeme dale.}
\xymatrix@C=3em{
f^{-1}(i) \ar@{~>}[d]   \ar@{-->}[r]^{\tilde {\rho_i}}
& 
{\tilde f}^{-1}(j) \ar@{~>}[d] 
\\
|f|^{-1}(i)  \ar[r]^{{\rho_i}}_\cong    & \ {|\tilde f|}^{-1}(j)\ .
}
\end{equation}
\end{definition} 

\begin{example}
The skeletal category $\sFin$ is cloven, 
as is the category $\Gr$ of
graphs and its modifications listed in diagram (37)
of~\cite{kodu}. Operadic categories of graphs are addressed in detail
in Section~\ref{co tri dny} of this article. 
\end{example}

\begin{nonexample}
\label{Tohle pisu 28.12. 2024 v Mercine.}
Recall that a standard, thin operadic category $\ttO$ is {\em ordered\/} if
the cardinality functor $|\dash| : \ttO \to \sFin$ factorizes through
the subcategory $\Delta \subset \sFin$ of finite ordinals and their
order-preserving maps. Non-unary ordered thin operadic categories,
such as
$\Delta$ itself, serve as generic examples of non-cloven thin operadic
categories. Indeed, take $T \in \ttO$ such that $|T| =
\underline n$, $n \geq 2$, and an automorphism $\sigma :  \underline n
\to \underline n$ which does not preserve order. Its lift
$\tilde \sigma$ of must satisfy $|\tilde \sigma| = \sigma$,
which is not possible, since $\sigma$ does not preserve the order of
$\underline n$. Below we give two examples of non-ordered thin operadic 
categories which are not cloven.
\end{nonexample}

\begin{nonexample}
\label{za necele tri hodiny}
We claim that 
Batanin's thin operadic category $\Omega_k$ of $k$-trees~\cite[\S 1.1]{duodel} is
cloven only if  $k =0$, 
and ordered only if $k \leq 1$. Indeed, the category $\Omega_0$ is the terminal
unary operadic category {\tt 1}, therefore (trivially) ordered and cloven. The
category $\Omega_1$ equals $\Delta$, which is ordered, therefore not
cloven by non-Example~\ref{Tohle pisu 28.12. 2024 v Mercine.}.

Let us analyze $\Omega_2$. 
Its objects are morphisms of $\Delta \subset \sFin$, 
i.e.~order-preserving maps $T: \underline m \to \underline
n$. Morphisms \hbox{$F: T' \to T''$} of objects of $\Omega_2$  
are commutative diagrams in the category of sets
\begin{equation}
\label{za 3 hodiny prohlidka}
\xymatrix@C=4em{
\underline m'  \ar[r]^\omega \ar[d]_{T'}   & \underline m'' \ar[d]_{T''} 
\\
\underline n' \ar[r]^\varsigma   & \underline n''
}
\end{equation}
such that 
\setlength{\leftmargini}{3em}
\begin{itemize}[topsep=0pt, partopsep=0pt, itemsep=0pt]
\item[(i)]
$\varsigma$ is order-preserving and
\item[(ii)]
for each $i \in \underline n'$, the restriction of $\omega$ to the
preimage $\inv{T'}i$ is order-preserving.
\end{itemize}
The cardinality of $T: \underline m \to \underline n  \in \Omega_2$ is $\underline
m\in \sFin$, and the morphism $F$ represented by~\eqref{za 3 hodiny prohlidka}
is mapped by the cardinality functor to $\omega : \underline m' \to \underline m''$. 
The objects of $\Omega_2$ can be viewed as trees with two
levels of vertices, whence the name.
For instance, the $2$-tree $T: \underline 4 \to \underline 2$ with
$T(1) = T(2) := 1$, $T(3) = T(4) := 2$ of cardinality 
$\underline 4$ is depicted as
\begin{center}
\includegraphics{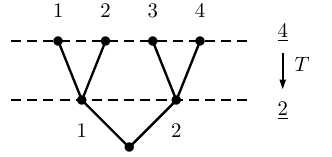}
\end{center}
It is easy to see that the automorphism $\sigma : \underline 4 \to
\underline 4$ given by $\sigma(1) := 2$, $\sigma(2) := 1$, $\sigma(3)
:= 4$ and $\sigma(4) := 3$ does not admit a lift, so $\Omega_2$ is not
cloven. To show that $\Omega_2$ is not ordered, consider a $2$-tree $S:
\underline 2 \to \underline 1$ and a morphism $F : T \to S$ given by
the morphism $\omega : \underline 4 \to \underline 2$ of leaves 
with $\omega(1) := 1$,
$\omega(2) := 2$, $\omega(3) := 1$ and
$\omega(2) := 2$. The map $|F| = \omega$ is not order-preserving. The
categories $\Omega_k$, $k \geq 3$, can be discussed similarly.
\end{nonexample}

\begin{nonexample}
\label{Dnes_na_muslicky.} 
  The thin operadic category $\Vines$ of vines  has the same objects as $\sFin$
  but morphisms $\underline m\to \underline n$ are isotopy classes of merging
  descending strings in $\mathbb{R}^3$ as in
\begin{center}
\includegraphics{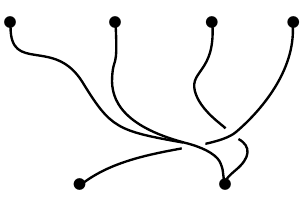}
\end{center}
cf.~\cite[Example~1.4]{env}. The cardinality $|\dash|: \Vines \to
\sFin$ is the identity on objects and maps $i \in \underline m$ to the
element of $\underline n$ connected with $i$ by a string.
It follows immediately from the definition that the subcategory of
automorphism of objects of cardinality $\underline n$ equals the
braided group ${\rm Br}_{n}$ on \hbox{$n$} elements. A cleavage thus
must be a functorial section  of the
projection $ {\rm Br}_{n}\twoheadrightarrow \Sigma_{n}$ of ${\rm
  Br}_{n}$ to the symmetric
group. Sections exist, but the functoriality cannot be achieved if $n
\geq 2$. This example was suggested to me by M.~Batanin.
\end{nonexample}

We have seen that there are two issues that can prevent the existence of a
cleavage. Either some automorphisms have no lifts as in
non-Examples~\ref{Tohle pisu 28.12. 2024 v Mercine.} and~\ref{za
  necele tri hodiny}, or the functoriality of lifts cannot be
achieved, as in non-Example~\ref{Dnes_na_muslicky.}.

\section{Cloven unitality}
\label{V sobotu jdu s Jarkou na Vystaviste.}

This section introduces the concept  of unitality for thick cloven
operadic categories. The following definition is to be compared with 
Definition~\ref{Poleti se zitra?}, which recalls the
unitality for standard, thin operadic categories.

\begin{definition}
\label{Pisu v Mercine}
Assume that  a family 
\begin{equation}
\label{Jsem posledni den v Cuernavace - uz je to dlouho.}
\big\{\bU_c \in \bO \ | \ c \in \pi_0(\bO) \big\}
\end{equation}
of local terminal objects of a thick operadic category $\bO$ 
is given, with $\bU_c$
belonging to the connected component~$c$, and such that $|\bU_c| =
\{1\}$ for each $c$.

The category  $\bO$ is {\em
  left unital\/} if, for each $\bT \in \bO$ with $|\bT| = \{\rada 1n\}$, the fiber
$\id_\bT^{-1}(1)$  of the identity $\id_\bT : \bT \to \bT$ 
equals $\bU_c$ with some $c \in \pi_0(\bO)$.  The category $\bO$ is {\em right
  unital\/} if the fiber functor ${\rm Fib}_1 : \bO/\bU_c \to \bO$ is the domain functor
for each $c \in \pi_0(\bO)$. The category $\bO$ is {\em unital\/} if it
is both left and right unital. 
\end{definition}

Let $\bO$ be a thick cloven operadic category 
equipped with a family~(\ref{Jsem posledni den v Cuernavace - uz je to
  dlouho.}). For each singleton
$\{x\}$ define $\bU^{\ \{x\}}_c$ as the target of the lift of the unique
isomorphism $\{1\} \to \{x\}$ in the lifting square
\begin{equation}
\label{remeslnici na chalupe}
\xymatrix@C=3em{\bU_c \ar@{~>}[d]\ar@{-->}[r]& 
\bU^{\ \{x\}}_c \ar@{~>}[d]
\\
\{1\} \ar[r]^\cong & \ \{x\}.
}
\end{equation}
Since $\bU^{\ \{x\}}_c$ is isomorphic to $\bU_c$, it is local terminal too. We
thus have, for each $c \in \pi_0(\bO)$, an action of the big groupoid
$\mathfrak G$ of
one-point sets on the local terminal objects in the connected
component~$c$. Notice that $\bU^{\ \{1\}}_c = \bU_c$

\begin{proposition}
\label{Prohlidka 30. prosince se blizi.}
Let\, $\bO$ be a thick cloven operadic category. 
If\, $\bO$ is left unital,
the fiber functor ${\it Fib}_{\{x\}} : \bO/\bU^{\ \{x\}}_c \to \bO$ is the domain
  functor for each $c \in \pi_0(\bO)$ and a singleton $\{x\}$. 
If\, $\bO$ is right unital, 
the fiber $\id_\bS^{-1}(x)$  of the identity $\id_\bS : \bS \to \bS$  
over $x\in |\bS|$ equals $\bU^{\ \{x\}}_{c_x}$ with some $c_x \in \pi_0(\bO)$, 
for each $\bS \in \bO$.

More generally, if\/ $\bO$ is right unital, 
$\sigma : X\to Y$ be an isomorphism of finite sets
and $\tilde \sigma : \bS \to \tilde \bS$ the lift in the lifting square
\[
\xymatrix@C=3em{ \bS \ar@{~>}[d]\ar@{-->}[r]^{\tilde \sigma}& 
\tilde \bS \ar@{~>}[d]
\\
X\ar[r]^{\sigma}_\cong & \ Y ,
}
\]
then, for each $x \in X$ and $y := \sigma(x)$, one has ${\tilde
  \sigma}^{-1}(y)= \inv{\id_{\bS}}x = \bU^{\   \{x\}}_{c_x}$ 
with some $c_x \in \pi_0(\bO)$. 
\end{proposition}

\begin{proof}
Suppose that $\bO$ is left unital and 
let $! : \bS \to \bU^{\ \{x\}}_c$ resp.~ $! : \bS \to \bU_c$ be the unique
maps to the appropriate local terminal objects.  They determine the upper square of
the diagram
\begin{equation}
\label{Nemohu se tam dovolat.}
\xymatrix@C=2em@R=1.2em{& 
\ar@{-->}[rr]^{!} \ar@{~>}'[d][dd]\bS&&
\bU_c^{\ \{x\}} \ar@{~>}[dd]
\\
\bS \ar@{-->}[ur]^{\id_\bS} \ar[rr]^(.6){!}  \ar@{~>}[dd] && 
\bU_c \ar@{-->}[ur]
\ar@{~>}[dd]
\\
&X \ar@{-->}'[r]^(.7){{!}}[rr] &&  \{x\}
\\
X \ar[rr]^{!}\ar[ur]^{\id_X}  && \ar[ur]^{!}
\{1\}
}
\end{equation} 
of the obvious maps and their lifts. 
The left face of~\eqref{Nemohu se tam dovolat.} is a
lifting square since cleavage preserves identities, and the right face
is a lifting square by the definition of $\bU^{\ \{x\}}_c$.
The associated lifting square~\eqref{7my den na chalupe} is
\[
\xymatrix@C=3em{{!}^{-1}(1) \ar@{~>}[d]\ar@{-->}[r]^{\tilde {\id_X}}& 
{!}^{-1}(x)\ar@{~>}[d]
\\
X\ar[r]^{\id_X} & \ X .
}
\]
Since the lift of the identity automorphism is the 
identity automorphism, ${!}^{-1}(1) = {!}^{-1}(x)$.  
Moreover, $!^{-1}(1) = \bS$ by the left unitality.
This proves the first part of the proposition.

Assume that $\bO$ is right unital.
Consider an arbitrary isomorphism of finite sets $\omega : A \to X$. 
The lift of the inverse of $\omega$ gives the
left square of the diagram
\begin{equation}
\label{Mam smlouvu az do konce roku 2027.}
\xymatrix@C=2em@R=1.2em{& 
\ar@{-->}[rr]^{\tilde \sigma} \ar@{~>}'[d][dd] \bS&&
\tilde \bS \ar@{~>}[dd]
\\
\bT \ar@{-->}[ur]^{\tilde \omega}  \ar[rr]^(.6){\id_\bT}  \ar@{~>}[dd] && 
\bT \ar@{-->}[ur]^{\tilde {\sigma
  \omega}}  
\ar@{~>}[dd]
\\
&X \ar'[r]^(.7){{\sigma}}[rr] &&Y
\\
A \ar[rr]^{\id_A}\ar[ur]^\omega  && \ A\, . \ar[ur]^{\sigma\omega} 
}
\end{equation}
Given $a \in A$, $x := \omega(a)$ and $y := \sigma(x)$, diagram~\eqref{Mam
  smlouvu az do konce roku 2027.} 
induces the lifting square
\begin{equation}
\label{Zitra jdeme s Jarkou na operu.}
\xymatrix@C=3em{ \ar@{~>}[d]\ar@{-->}[r] \inv{\id_\bT}a & \tilde \sigma^{-1}(y)
 \ar@{~>}[d]
\\
\{a\} \ar[r]^\cong & \ \{y\}.
}
\end{equation}
Choosing $A := \{\rada 1n\}$ and $\omega :  \{\rada 1n\} \to X$ an
arbitrary map such that $\omega(1) = x$, diagram~\eqref{Zitra jdeme s
  Jarkou na operu.} with $a := 1$ gives
\begin{subequations}
\begin{equation}
\label{drevene auticko}
\xymatrix@C=3em{ \ar@{~>}[d]\ar@{-->}[r] \bU_c := \inv{\id_\bT}1  & \tilde \sigma^{-1}(y)
 \ar@{~>}[d]
\\
\{1\} \ar[r]^\cong & \ \{x\} ,
}
\end{equation}
so $\tilde \sigma^{-1}(y) = \bU_c^{\ \{x\}}$ by~\eqref{remeslnici na chalupe}.
Likewise, diagram~\eqref{Mam smlouvu az do konce roku 2027.} with $\bT
= \bS$, $A = X$ and $\omega = \id_X$ induces, for $a =x \in X$ and $y := \sigma(x)$,
\begin{equation}
\label{Musim se donutit jit si zabehat. Snad koleno vydrzi.}
\xymatrix@C=3em{ \ar@{~>}[d]\ar@{-->}[r] \inv{\id_{\tilde \bS}}x & \tilde \sigma^{-1}(y)
 \ar@{~>}[d]
\\
\{x\} \ar[r]^\id & \ \{x\}.
}
\end{equation}
\end{subequations}
Diagrams~\ref{drevene auticko} and~\eqref{Musim se donutit jit si
  zabehat. Snad koleno vydrzi.} together imply the rest of the proposition.
\end{proof}

Recall an important class of morphisms of a standard thin  
unital operadic category, formed by   {\/\em
  quasibijections\/}, which are morphisms
whose all fibers are the chosen local terminal
objects~\hbox{\cite[Section~2]{env}}. 
A similar definition makes sense also in the thick case, with the role of the
chosen local terminal objects played by $\bU^{\ \{x\}}_c$ for $c \in
\pi_0(\bO)$ and singletons $\{x\}$. Proposition~\ref{Prohlidka
  30. prosince se blizi.} has an obvious

\begin{corollary}
Lifts of isomorphisms in a cloven unital operadic category are quasibijections.
\end{corollary}

\section{Operads}

Operadic categories which  we have discussed so far form the initial stage
of the triad
\begin{equation}
\label{Musim jit na druhy odber.}
\xymatrix{
\boxed{\ \hbox {\ operadic categories}} \ \ar@{=>}[r] 
&
\ \boxed{\hbox {\ operads}} \ \ar@{=>}[r]
&
\boxed{\hbox {algebras}} 
}
\end{equation}
in which ``$A \Longrightarrow B$'' must be read as \ ``$A$ governs
$B$.'' In this section we discuss the next one.

We start by notation that will simplify future exposition.
Let $\bO$ be a cloven thick operadic category and  $\boE =
\{\boE(\bT)\}_{\bT \in \bO}$ a collection  of objects of a
symmetric monoidal category $\ttV$. We say that $\boE$ is a {\/\em
  cloven module\/} if the lifts $\tilde \sigma : \bS \to \bT$ of
isomorphisms $\sigma : X \to Y$ in the lifting square
\begin{equation}
\label{Nemohu najit tu propisovacku.}
\xymatrix@C=4em{\bS \ar@{-->}[r]^{\tilde \sigma}_\cong \ar@{~>}[d]& \bT  \ar@{~>}[d]
\\
X \ar[r]^\sigma_\cong &\ Y
}
\end{equation} 
act contravariantly and   
functorially on $\boE$ via isomorphisms
$\boE(\sigma) : \boE(\bT) \stackrel\cong\longrightarrow \boE(\bS)$.

For a morphism $\boldf : \bS \to \bT$ with fibers $\{ \bF_{\!a} \}_{a \in
  |\bT|}$, we denote by $\boE(\boldf)$ the ``unordered tensor product'' 
$\bigotimes_{a \in |\bT|} \boE(\bF_{\!a})$ given by the colimit
in~\cite[Definition~I.1.58]{markl-shnider-stasheff:book}. 
In the situation of diagram~\eqref{Ve ctvrtek letim do Dvora.} we have 
a~natural isomorphism
\begin{equation}
\label{Az 30. prosince se to dozvim.}
\boE(\rho) :=  \bigotimes_{a \in A} \boE(\rho_a):
\boE(\btf) \stackrel\cong\longrightarrow \boE(\boldf),
\end{equation}
where $\rho_a : |\boldf|^{-1}(a) \to  |\btf|^{-1}(\sigma(a))$ is
the induced isomorphism between the set-theoretic preimages. Regarding
notation, we hope that there
will be no confusion between \hbox{$\boE(\sigma) : \boE(\bT) \to
\boE(\bS)$} and the morphism in~\eqref{Az 30. prosince se to dozvim.};
the meaning will always be clear from the context.

\begin{definition}
\label{Na Safari oper strasne pocasi.}
Let $\bO$ be a thick cloven operadic category as in Definition~\ref{8
  dni na chalupe!}.  A (nonunital) {\/\em operad\/} is
a cloven module  $\boP = \{\boP(\bT)\}_{\bT \in \bO}$ with the
composition law
\begin{equation}
\label{structure operations}
\bmu_\boldf : \boP(\bT)\ \ot\ \boP(\boldf)
\longrightarrow \boP(\bS)
\end{equation}
specified for each morphism $\boldf : \bS \to \bT$ and subject to the
following properties.
\setlength{\leftmargini}{2em}
\begin{itemize}[topsep=0pt, partopsep=0pt, itemsep=0pt]
\item[(i)] 
Associativity. Let $\bT \stackrel \boldf\to \bS \stackrel \bg\to \bR$ be morphisms in
$\bO$, and $\bh := \bg \boldf : \bT \to \bR$ as in~(\ref{Kure}).  Then the
following diagram of composition laws of $\boP$ combined with the
canonical isomorphisms of iterated products in $\ttV$ commutes:
\begin{subequations}
\begin{equation}
\label{Dnes si pujdu koupit tri venecky a slehacku.}
\xymatrix@C = 2em@R = .4em{
\ar@/^2.5ex/[rrd]^(.56){\id \ \ot \bigotimes_{r}\bmu_{\boldf_r}}
      \ar[dd]_(.45){\bmu_\bg  \ot \id}
\displaystyle\bigotimes_{r
        \in |\bR|} 
    \boP(\bR) \ot  \boP(\bg) \ot  \boP(\boldf_r) & &    
      \\  & &  \ar@/^/[dl]_{\bmu_\bh}
    \boP(\bR) \ot  \boP(\bh)\ .
\\
      \ar[r]^(.77){\bmu_\boldf}{\rule{0pt}{2em}}   \boP(\bS) \ot
\displaystyle\bigotimes_{r \in |\bR|}
      \boP(\boldf_r)  \cong    \boP(\bS) \ot \boP(\boldf)&
 \boP(\bT)&
    }
\end{equation}
The isomorphism $\bigotimes_{r \in |\bR|} \boP(\boldf_r) \cong
\boP(\boldf)$ used in the bottom line holds by~\eqref{karneval_u_retardacku}. 

\item[(ii)]  
Compatibility with the action. In the situation
 of~\eqref{Ve ctvrtek letim do Dvora.}, 
we require the commutativity of
\begin{equation}
\label{Vcera jsem byl se Sambou v Krkonosich.}
\xymatrix@C=4em{\boP(\bT)\ \ot\ \boP(\boldf)
\ar[r]^(.6){{\bmu_\boldf}} \ar@{<-}[d]_{\boP(\sigma)\ \ot\ \boP(\rho)}
& \boP(\bS) \ar[d]^{\boP(\rho)}
\\
\boP(\tilde \bT) \ot  \boP(\btf)
\ar[r]^(.6){{\bmu_{\btf}}} & \boP(\tilde\bS).
}
\end{equation}
\end{subequations}
\end{itemize}
\end{definition}

The following definition assumes 
that $\bO$ is a cloven thick unital operadic category with
collection~\eqref{Jsem posledni den v Cuernavace - uz je to dlouho.}
of chosen local terminal objects. Each family of morphisms
\begin{subequations}
\begin{equation}
\label{Pujdu koupit orisky.}
\{\eta_c : \unit
\longrightarrow \boP(\bU_c) \ | \ c \in \pi_0(\bO)\},
\end{equation}
where $\unit$ is the monoidal unit of $\ttV$, generates the morphisms
\begin{equation}
\label{Poletime zitra nebo v nedeli?}
\bha_c^{\{x\}} :  
\unit \stackrel {\eta_c}\longrightarrow \boP(\bU_c) \stackrel\cong\longrightarrow 
\boP(\bU^{\ \{x\}}_c) 
\end{equation}
\end{subequations}
specified for each $c \in \pi_0(\bO)$ and a singleton $\{x\}$; 
the isomorphism $\boP(\bU_c) \stackrel\cong\to
\boP(\bU^{\ \{x\}}_c)$  is given by the cloven module action of $\bU_c \to \bU_c^{\,
  x}$ in~\eqref{remeslnici na chalupe}. Notice that $\bha_c^{\{x\}} = \eta_c$.

\begin{definition}
\label{Nekde se toula.}
An operad $\boP$ as in Definition~\ref{Na Safari oper strasne pocasi.}
is {\/\em unital\/} if it is equipped with morphisms~\eqref{Pujdu koupit orisky.} and the following conditions
are satisfied for each $\bT \in \bO$.
\setlength{\leftmargini}{2em}
\begin{itemize}[topsep=0pt, partopsep=0pt, itemsep=0pt]
\item [(i)]
The diagram
\[
\xymatrix{\boP (\bU_c) \ot \boP(\bT) \ar[r]^(.6){\bmu_!} &
  \boP(\bT)
\ar@{=}[d]
\\
\unit \ \ot \boP(\bT) \ar[u]^{\eta_c \ot \id} \ar[r]^(.56)\cong &\boP(\bT),
}
\]
in which $\bmu_!$ is the operation associated to the
unique morphism $! : \bT \to \bU_c$, commutes.
\item [(ii)]
The diagram
\begin{equation}
\label{Zitra jdu na prohlidku.}
\xymatrix{\boP(\bT) \otimes \bigotimes_{c_x}    
\boP (\bU^{\ \{x\}}_{c_x})  \ar[r]^(.69){\bmu_{\id}} &
  \boP(\bT)
\ar@{=}[d]
\\
\boP(\bT)\ \ot \ \unit^{\ot |\bT|} \ar[u]^{\id \ot 
\bigotimes_{c_x} \bha^{\{x\}}_{c_x}} \ar[r]^(.6)\cong &\ \boP(\bT),
}
\end{equation}
where $\bU_{c_x}^{\ \{x\}}$ is the 
fiber of \/ $\id_\bT : \bT \to \bT$ over $x$, and $x$ runs over the
set $|\bT|$, commutes.
\end{itemize}
\end{definition}

It follows from the compatibility with the
cloven module action required in (ii) of Definition~\ref{Na Safari oper strasne
    pocasi.}  that it suffices to
  verify  the properties (i) and (ii) only for $\bT \in \bO$
with $|\bT| = \{\rada 1n\}$, $n \geq 1$. The compatibility 
also implies that a stronger version of the diagram in (i)
of Definition~\ref{Nekde se toula.} holds, namely
\begin{equation}
\label{Dnes jdeme na operu.}
\xymatrix{\boP (\bU_c^{\ \{x\}}) \ot \boP(\bT) \ar[r]^(.62){\bmu_!} &
  \boP(\bT)
\ar@{=}[d]
\\
\unit \ \ot \boP(\bT) \ar[u]^{\bha_c^{\{x\}} \ot \id} \ar[r]^(.56)\cong &\boP(\bT),
}
\end{equation}
where $!: \bT \to \bU^{\ \{x\}}_c$ is the unique morphism to a local
terminal object.
To verify the commutativity of~\eqref{Dnes jdeme na operu.}, notice that (ii) of
Definition~\ref{Na Safari oper strasne pocasi.} applied to  
\[
\xymatrix@C=1.8em@R=1.4em{& 
\ar@{-->}[rr]^{!} \ar@{~>}'[d][dd] \bT &&
 \ar@{~>}[dd] \bU_c^{\ \{x\}}
\\
\bT \ar@{-->}[ur]^{\id_\bT}  \ar[rr]^(.6){!}  \ar@{~>}[dd] && 
\bU_c \ar@{-->}[ur]^{!}  
\ar@{~>}[dd]
\\
&|\bT|\ar@{-->}'[r]^(.7){!}[rr] &&\, \{x\}.
\\
|\bT| \ar[rr]^{!}\ar[ur]^{\id_{|\bT|}}  &&\  \{1\} \ar[ur]^{!}
}
\]
implies the commutativity of the square in 
\[
\xymatrix@R=1em{
& \boP(\bU_c) \ot \boP(\bT)  \ar[rr]^(.6){\bmu_!}   && \boP(\bT)
\\
\unit \ot \boP(\bT) \ar@/^1pc/[ur]^{\eta_c \ot \id}\ar@/_1pc/[dr]^{\bha_c^{\{x\}} \ot \id}
\\
& \boP(\bU_c^{\ \{x\}}) \ot \boP(\bT) \ar[rr]^(.6){\bmu_!}  
\ar[uu]_{\boP(!) \ot \id} 
&&\, \boP(\bT) , \ar@{=}[uu]
}
\]
while the triangle is commutative by the definition of $\bha_c^{\{x\}}$.

It can be proven that the cloven module action of a unital operad
$\boP$ over a thick unital operadic category $\bO$ is related with
family~\eqref{Poletime zitra nebo v nedeli?}
via the diagram
\begin{equation}
\label{Pozitri jedeme no Mercina na Vanoce.}
\xymatrix{
\boP(\bT) \ot \bigotimes_{x \in X} \boP(\bU^{\ \{x\}}_{c_x})
\ar[r]^(.7){\bmu_{\tilde \sigma}} & \boP(\bS)
\\
\boP(\bT) \ot  \bigotimes_{x \in X} \unit 
\ar[u]^{\id \ot \bigotimes \bha^{\{x\}}_{c_x}} \ar[r]^(.6)\cong
& \boP(\bT) \ar[u]_{\boP(\sigma)},
}
\end{equation}
where $\tilde \sigma: \bS \to \bT$ is as in~(\ref{Nemohu najit tu
  propisovacku.}), $\bU_{c_x}^{\ \{x\}}$ is the 
fiber of \/ $\id_\bS : \bS \to \bS$ over $x$, which equals the fiber
$\inv{\tilde \sigma}y$ over $y := \sigma(x)$ by the second part of
Proposition~\ref{Prohlidka 30. prosince se blizi.}, 
and where $x$ runs over $X$.

Morphisms~\eqref{Poletime zitra nebo v nedeli?} 
were generated from the unit maps~(\ref{Pujdu koupit orisky.}) using the cloven module structure of $\oP$. The following
proposition shows that, vice versa, the
module structure of $\boP$ is determined by a family
\begin{equation}
\label{Jdu tam ve 13:30.}
\bha_c^{\{x\}} :  
\unit \longrightarrow 
\boP(\bU^{\ \{x\}}_c), \ c \in \pi_0(\bO), \ \{x\} \hbox { a sindleton,} 
\end{equation}
of morphisms as in~\eqref{Poletime zitra nebo v nedeli?},
suitably compatible with the structure operations of $\boP$.

\begin{proposition}
\label{Uz mi klici papricky.}
Let $\bO$ be a thick unital operadic category and $\boP = \{\boP(\bT)\}_{\bT \in \bO}$ a 
collection with composition 
laws~\eqref{structure operations} which are associative as required in
(i) of Definition~\ref{Na Safari oper strasne pocasi.}. Assume
moreover that $\boP$ comes with a
family~\eqref{Jdu tam ve 13:30.} satisfying~(\ref{Zitra jdu na
  prohlidku.}) and~(\ref{Dnes
  jdeme na operu.}).  Diagrams~(\ref{Pozitri jedeme no Mercina na
  Vanoce.}) then
define a cloven module action that satisfies (ii) of Definition~(\ref{Na Safari oper strasne pocasi.}).
\end{proposition}

As we will not need the above proposition, we omit its 
technical, but straightforward, proof.

\begin{remark}
\label{Jezecek Bodlinek a Myska}
If we take the hypotheses of
Proposition~\ref{Uz mi klici papricky.} as the
definition of unitality, unital operads over thick operadic categories
are automatically  equipped with a cloven module action that satisfies  
(ii) of Definition~\ref{Nekde se toula.}.
\end{remark}

The thin analog of a cloven module over a thin operadic category
$\ttO$ obvious. It is a collection $\oE =
\{\oE(T)\}_{T \in \ttO}$ such that the lifts $\tilde \sigma : S \to
T$ in the lifting square
\begin{equation}
\label{Asi se te Cine nevyhnu.}
\xymatrix@C=4em@R=2em{S \ar@{-->}[r]^{\tilde \sigma} \ar@{~>}[d]& T  \ar@{~>}[d]
\\
\underline n \ar[r]^\sigma_\cong &\ \underline n
}
\end{equation}
act contravariantly and   
functorially on $\oE$ via isomorphisms
$\oE(\sigma) : \oE(T) \stackrel\cong\longrightarrow \oE(S)$.
An analog of the notation in~\eqref{Az 30. prosince se to
  dozvim.} is also clear. We are ready for

\begin{definition}
\label{Zacina byt poradna zima, ale ne takova aby zamrzlo Hradistko.}
A (nonunital) {\/\em operad\/} over a thin cloven operadic category $\ttO$ is a cloven 
module  $\oP = \{\oP(T)\}_{T \in \ttO}$ with the
composition law
\[
\mu_f : \oP(T)\ \ot\ \oP(f)
\longrightarrow \oP(S)
\]
specified for each morphism $f : S \to T$. We moreover demand the following two
properties.
\setlength{\leftmargini}{2em}
\begin{itemize}[topsep=0pt, partopsep=0pt, itemsep=0pt]
\item[(i)] 
Associativity. The diagram in (i) of~\cite[Definition 1.11]{duodel},
which was the blueprint for its thick version~(\ref{Dnes si
  pujdu koupit tri venecky a slehacku.}), commutes for each chain $T
\stackrel f\to S \stackrel g\to R$ of maps in $\ttO$. 
\item[(ii)]  
Compatibility with the action. In the situation 
of~\eqref{Zitra se budeme divat na manzele Stodolovy.}, 
we require the commutativity~of
\[
\xymatrix@C=4em{\oP(T)\ \ot\ \oP(f)
\ar[r]^(.6){{\mu_f}} \ar@{<-}[d]_{\oP(\sigma)\ \ot\ \oP(\rho)}
& \oP(S) \ar[d]^{\oP(\rho)}
\\
\oP(\tilde T) \ot  \boP(\tilde f)
\ar[r]^(.6){{\mu_{\tilde f}}} & \oP(\tilde S).
}
\]
\end{itemize}
\end{definition}

Suppose that $\ttO$ is unital, and the operad $\oP$ is unital in the
standard sense, that is,
there are unit maps $\eta_c : \unit \to \oP(U_c)$, $c \in  \pi_0(\ttO)$,
satisfying items (ii) and (iii) of~\cite[Definition 1.11]{duodel}.
It can be easily shown that the cloven module action
is related with
the unit maps of $\oP$ by the diagram
\begin{equation}
\label{Dnes ve 12:55 sraz s Jarkou.}
\xymatrix{
\oP(T) \ot \bigotimes_{i \in |S|} \oP(U_{c_i})
\ar[r]^(.7){\mu_{\tilde \sigma}} & \oP(S)
\\
\oP(T) \ot  \bigotimes_{i \in |S|} \unit 
\ar[u]^{\id \ot \bigotimes \eta_{c_i}} \ar[r]^(.6)\cong
& \oP(T) \ar[u]_{\oP(\sigma)},
}
\end{equation}
where $\tilde \sigma: S \to T$ is as in~(\ref{Asi se te Cine
  nevyhnu.}) and $U_{c_i}$ is the 
fiber of \/ $\id_S : S \to S$ over $i$. Diagram~\ref{Dnes ve 12:55
  sraz s Jarkou.} can be used as 
a~definition of the cloven module action. The following proposition with
straightforward proof has to be compared with Remark~\ref{Jezecek Bodlinek a Myska}.

\begin{proposition}
Unital $\ttO$-operads in the sense of the original
definition~\cite[Definition 1.11]{duodel} are automatically $\ttO$-operads in
Definition~\ref{Zacina byt poradna zima, ale ne takova
  aby zamrzlo Hradistko.}. 
\end{proposition}

\section{Algebras}
\label{MR byla v poradku!}

This section is devoted to the lowest, terminal level of the
triad~\eqref{Musim jit na druhy odber.}.
In the ``standard'' definition
of operads~\cite[Definition~1.20]{duodel}, algebras appear as collections indexed by the
connected components of the background operadic category. They are subject
to an associativity axiom, which uses the notions of the source and the
target of an object. We therefore open this section with a thick version
of these notions.

We define the {\/\em source\/} $\bs(\bT)$ 
of an object $\bT \in \bO$ to be the list 
of the connected components of the fibers of the identity $\id_{\bT} : \bT
\to \bT$. The {\/\em target\/} $\bt(\bT)\in \pi_0(\bO)$ is the connected component
of $\bT$.

\begin{lemma}
\label{Mel bych to vydrzet.}
Let $\bS, \bT$ be objects of\/ $\bO$. 
The lift $\tilde \sigma :\bS \to \bT$ in 
the lifting square~\eqref{Nemohu najit tu propisovacku.} acts functorially on the
sources by an isomorphism
$\bs(\sigma) : \bs(\bS) \to \bs(\bT)$. 
\end{lemma}

\begin{proof}
Embed  $\tilde \sigma :\bS \to \bT$ to the diagram
\[
\xymatrix@C=2.1em@R=1.4em{& 
\ar@{-->}[rr]^{\id_\bT} \ar@{~>}'[d][dd] \bT&&
\bT \ar@{~>}[dd]
\\
\bS \ar@{-->}[ur]^{\tilde \sigma}  \ar[rr]^(.6){\id_\bS}  \ar@{~>}[dd] && 
\bS \ar@{-->}[ur]^{\tilde \sigma}  
\ar@{~>}[dd]
\\
&Y\ar@{-->}'[r]^(.7){{\id_Y}}[rr] &&\, Y.
\\
X \ar[rr]^{\id_X}\ar[ur]^\sigma_\cong  &&X \ar[ur]^\sigma_\cong 
}
\]
The associated lifting squares~\eqref{7my den na chalupe} provide the morphism
$\tilde{\sigma_{\! x}} : \id_\bS^{-1}(x) \to
\id_\bT^{-1}(\sigma(x))$, for each $x \in X$. 
Therefore the fiber of $\id_\bS$ over $x$ belongs to the same
component of $\bO$ as the fiber of $\id_\bT$ over $\sigma(x)$.
The assignment
\begin{equation}
\label{Dnes dopopledne uz mam odpracovano.}
\bs(\sigma) : \hbox {the component of }\  \id_\bS^{-1}(x)
\longmapsto \hbox {the component of }\ \id_\bT^{-1}(\sigma(x)),\ x \in X,
\end{equation}
defines the requisite action $\bs(\sigma) : \bs(\bS) \to \bs(\bT)$. 
\end{proof}

Given a collection $\bA = \{\bA_c \in \ttV \ | \ c
\in \pi_0(\bO)\}$ indexed by the set of connected components of~$\bO$,
we denote for $\bT \in \bO$
\[
\bA_{\bs(\bT)} := \bigotimes_{c \in \bs(\bT)} \bA_c.
\] 

\begin{definition}
\label{Vezme to?}
An {\/\em algebra\/} for an operad $\boP$ is a collection $\bA = \{\bA_c\}_{c
\in \pi_0(\bO)}$ of objects of $\ttV$, along with the actions
\[
\bal_\bT : \boP(\bT) \ot  \bA_{\bs(\bT)}  \longrightarrow \bA_{\bt(\bT)}, \ \bT \in \bO.
\]
We assume the following three properties.
\begin{subequations}
\setlength{\leftmargini}{2em}
\begin{itemize}[topsep=0pt, partopsep=0pt, itemsep=0pt]
\item[(i)] Associativity. For each morphism $\boldf : \bS \to \bT$
with fibers $\bF_{\! a}$, $a \in |\bT|$, we require the commutativity of 
\begin{equation}
\label{Sel jsem spat ve tri hodiny.}
\xymatrix@C=4em{
\boP(\bT) \ot \bigotimes_a\big( \boP(\bF_{\!a}) \ot  \bA_{\sou(\bF_{\!a})}\big)
\ar[r]^(.58){\id\ \ot \bigotimes_a \bal_{\bF_{\!a}}} \ar[d]_{\cong}
 & \oP(\bT) \ot \bigotimes_a \bA_{\tar( \bF_{\!a})} 
\ar[r]_(.56)\cong^(.56){\boxed{$\scriptsize 1$}}
& \boP(T) \ot \bA_{\sou(\bT)}\ar[d]^(.6){\bal_{\bT}}
\\
\ar@{=}[d]_{\bmu_\boldf \ \ot \id^{\ot |\bT|}}
\boP(\bT) \ot \boP(\boldf) \ot \bigotimes_a \bA_{\sou(\bF_a)}
&
&  \bA_{\tar( \bT)} \ar@{=}[d]^{\boxed{$\scriptsize 2$}}
\\
\ar[r]^{\boxed{$\scriptsize 3$}}_\cong
\boP(\bS) \ot \bigotimes_a \bA_{\sou(\bF_{\!a})} & \boP(\bS) \ot
\bA_{\sou(\bS)} 
\ar[r]^{\bal_{\bS}}
 &\ \bA_{\tar(\bS)} . 
}
\end{equation}
\end{itemize}
In the above diagram, \boxed{$\scriptsize 1$} uses the
isomorphism between the connected components of the fibers~$\bF_{\! a}$
and the source of $\bT$ implied by the existence of the induced maps
$\boldf_a : \bF_{\! a} \to \id^{-1}_\bT(a)$, $a \in |\bT|$,
guaranteed by (i) of Definition~\ref{Uz treti tyden je strasliva
  zima.} applied to
\[
\xymatrix@C = +1em@R = +1em{
\bS      \ar[rr]^\boldf \ar[dr]_\boldf & & \bT \ar[dl]^{\id_\bT}
\\
&\ \bT . &
} 
\]
Since $\bS$ and $\bT$ are connected by $\boldf$, they belong to the
same component of $\bO$, thus $\tar(\bS) = \tar(\bT)$, which explains 
\boxed{$\scriptsize 2$}\ . By (ii) 
of Definition~\ref{Uz treti tyden je strasliva
  zima.} applied to
\[
\xymatrix@C = +1em@R = +1em{
\bS      \ar[rr]^{\id_\bS} \ar[dr]_\boldf & & \bS \ar[dl]^{\boldf}
\\
&\ \bT ,&
} 
\]
the set of fibers of $\id_{\bF{\!a}}$, $a \in |\bT|$, is the same as
  the set of fibers of $\id_\bS$, which explains \boxed{$\scriptsize
    3$}\ .
\setlength{\leftmargini}{2em}
\begin{itemize}[topsep=0pt, partopsep=0pt, itemsep=0pt]
\item[(ii)]
Compatibility with the cloven action. Given the lifting
square~\eqref{Nemohu najit tu propisovacku.}, the diagram 
\begin{equation}
\label{Udelam si vajicka.}
\xymatrix@C=4em{
\boP(\bS) \ot \bA_{\bs(\bS)}
  \ar@{<-}[d]_{\boP(\sigma) \ot \bA_{\bs(\sigma)}}^\cong
\ar[r]^(.6){\bal_\bS} & \ar@{=}[d]
  \bA_{\bt(\bS)}
\\
\boP(\bT) \ot \bA_{\bs(\bT)}  \ar[r]^(.6){\bal_\bT}&
\  \bA_{\bt(\bT)},
}
\end{equation}
where $\bA_{\bs(\sigma)}:  \bA_{\bs(\bS)} \stackrel\cong\longrightarrow  
\bA_{\bs(\bT)}$ is the isomorphism 
given by the permutation of the factors according
to $\bs(\sigma)$, commutes.
\item[(iii)]
If the operad $\boP$ 
is unital, then we moreover assume that, for
each $c \in \pi_0(\bO)$, the diagram
\begin{equation}
\label{Dnes jsem toho moc nenaspal.}
\xymatrix@C=3.5em{\boP(\bU_c) \ot \bA_c  \ar[r]^(.6){\bal_{\bU_c}} & \bA_c
\\
\unit \ \ot \bA_c   \ar[u]^{\eta_c \ot \id} \ar[r]^\cong & \bA_c \ar@{=}[u]
}
\end{equation}
commutes.
\end{itemize} 
\end{subequations}
\end{definition}

In the following proposition, $\boP$ will be an unital operad over a
thick cloven unital operadic category 
$\bO$, and $\eta_c$ will be as in~\eqref{Pujdu koupit orisky.}.

\begin{proposition}
If $\boP$ is unital, then (ii) of Definition~\ref{Vezme to?} holds
automatically, and the following generalization
\[
\xymatrix@C=3.5em{
\boP(\bU^{\, \{x\}}_c) \ot \bA_c  \ar[r]^(.6){\bal_{\bU^{\{x\}}_c}} & \bA_c
\\
\unit \ \ot \bA_c \ar[u]^{\bha^{\{x\}}_c \ot \id} \ar[r]^\cong 
& \bA_c \ar@{=}[u]
}
\]
of~\eqref{Dnes jsem toho moc nenaspal.},
where $\bha^{\{x\}}_c$ is as in~\eqref{Poletime zitra nebo v nedeli?},
commutes.
\end{proposition}

\begin{proof}
The first part is a consequence of the relation between the cloven module
action on $\boP$ and its unital structure expressed by the
commutativity of~\eqref{Pozitri jedeme no Mercina na Vanoce.}, 
The second part follows from  the
definition~\eqref{Poletime zitra nebo v nedeli?} of $\bha^{\{x\}}_c$ and 
the commutativity of the diagram
\[
\xymatrix@C=3.5em{
\boP(\bU^{\,\{x\}}_c) \ot \bA_c  \ar[r]^(.6){\bal_{\bU^{\{x\}}_c}} & \bA_c
\\
\ar[u]^\cong
\boP(\bU_c) \ot \bA_c  \ar[r]^(.6){\bal_{\bU_c}} & \bA_c \ar@{=}[u]
}
\]
implied by (ii) of Definition~\ref{Vezme to?}.
\end{proof}

Let us briefly address algebras in the context of a thin cloven operadic
background category~$\ttO$.
For $T \in \ttO$ denote by $s(T)$ the list of connected
components of the fibers of $\id_T : T \to T$, and by $t(T)$ the
connected component of $T$. As in Lemma~\ref{Mel bych to vydrzet.} we
notice that the lifts $\tilde \sigma$  in~\eqref{Asi se te Cine
  nevyhnu.} act functorially by isomorphisms $s(\sigma) : s(S) \to s(T)$ of the sources.

\begin{definition}
\label{Jaruska uvarila brokolici.}
An {\/\em algebra\/} for an $\ttO$-operad $\oP$ is a collection $A = \{A_c\}_{c
\in \pi_0(\ttO)}$, equipped with the actions
\[
\alpha_T : \oP(T) \ot  A_{s(T)}  \longrightarrow A_{t(T)}, \ T \in \ttO,
\]
which satisfy the following three axioms.
\setlength{\leftmargini}{2em}
\begin{itemize}[topsep=0pt, partopsep=0pt, itemsep=0pt]
\item[(i)] Associativity. 
A thin analog of diagram~\eqref{Sel jsem spat ve tri
  hodiny.}, which actually appeared
in Definition~48 of~\cite{bivariant}, commutes for each morphism $S\to T$.
\item[(ii)]
Compatibility with the cloven action. An obvious thin analog of diagram~(\ref{Udelam si
  vajicka.}), namely
\[
\xymatrix@C=4em{
\oP(S) \ot A_{s(S)}
  \ar@{<-}[d]_{\oP(\sigma) \ot A_{s(\sigma)}}^\cong
\ar[r]^(.6){\alpha_S} & \ar@{=}[d]
  A_{t(S)}
\\
\oP(T) \ot A_{s(T)}  \ar[r]^(.6){\alpha_T}&
\  A_{t(T)},
}  
\]
commutes  for each lifting
square~\eqref{Asi se te Cine
  nevyhnu.}.
\item[(iii)]
If $\oP$ is unital, we assume that the diagram
\[
\xymatrix@C=3.5em{\oP(U_c) \ot A_c  \ar[r]^(.6){\alpha_{U_c}} & A_c
\\
\unit \ \ot A_c   \ar[u]^{\eta_c \ot \id} \ar[r]^\cong & A_c \ar@{=}[u]
}
\]
commutes for
each $c \in \pi_0(\ttO)$,.
\end{itemize} 
\end{definition}

As in the thick case, the compatibility (ii) is automatic if $\oP$ is unital. 
A $\oP$-algebra is then the same as a morphism from
$\oP$ to a suitable colored endomorphism operad, so our definition agrees with
Definition~1.12 of~\cite{duodel}.

\section{The equivalences}
\label{Dnes je Mikulase, a Mikulasek s Lachtankem jeste spinkaji.}

Operadic functors were introduced, in the context of the standard,
thin operadic categories, \hbox{in~\cite[page~1635]{duodel}}. 
Among other properties,  
operadic functors were required to commute with the cardinality
functors, to send fiber to fibers, and to send the chosen local terminal objects to
the chosen local terminal objects in the unital case. A modification
to the cloven case is given in

\begin{definition}
An operadic functor $F:\ttO' \to \ttO''$ between thin cloven operadic
categories is {\/\em cloven\/} if the images
\begin{subequations}
\begin{equation}
\label{Zitra}
\xymatrix@C=4em@R=2em{F(S) \ar@{-->}[r]^{F(\tilde \sigma)} \ar@{~>}[d]& F(T)  \ar@{~>}[d]
\\
\underline n \ar[r]^\sigma_\cong &\ \underline n
}
\end{equation}
of the lifting squares
\begin{equation}
\label{MR}
\xymatrix@C=4em@R=2em{S \ar@{-->}[r]^{\tilde \sigma} \ar@{~>}[d]& T  \ar@{~>}[d]
\\
\underline n \ar[r]^\sigma_\cong &\ \underline n
}
\end{equation}
\end{subequations}
in~$\ttO'$ are lifting squares in $\ttO''$. Operadic functors between thick
operadic categories and their cloven versions are defined analogously.
\end{definition}

Notice that, since $F$ commutes with the cardinality functor,
$|F(\tilde \sigma)| = |\tilde \sigma| = \sigma$, the diagram
in~\eqref{Zitra} is indeed the $F$-image of~\eqref{MR}.

Let us denote by $\copcat$ the category of thin cloven operadic
categories and their cloven operadic functors, and by $\COPCAT$ its
thick version. The aim of this section is to establish 
the central result of this article, namely

\begin{theorem}
\label{Dnes jedu do Kolina a zitra letim do Dvora.}
There exists a natural correspondence \ $\E : \ttO \leftrightarrow
\bO: \R$ (extension and restriction) which extends to an equivalence
of the categories $\copcat$ and \/ $\COPCAT$.
The categories of operads and their
algebras over the corresponding operadic
categories are naturally equivalent.

The thin operadic category \/ $\ttO$ is unital if and only if the corresponding
thick operadic category \/ $\bO$ is unital. If this is the case, then the above
equivalence restricts to an equivalence of the categories of unital
operads and their algebras. 
\end{theorem}

We denote, as before, by $\Fin$ the  category of finite
sets and by $\sFin$ its skeletal subcategory of 
finite ordered sets $\underline n :=
\{1,\ldots,n\}$, $n \geq 0$. The blackboard $\bO$ will serve as the generic
name of a~thick cloven operadic category, the typewriter $\ttO$ the
generic name of a thin cloven one.

The rest of this section is devoted to the proof of Theorem~\ref{Dnes
  jedu do Kolina a zitra letim do Dvora.}. 
First we define the restriction and the extension. Then, in
Proposition~\ref{Vratil jsem se od doktora Reicha.} we prove that they
are inverse to each other and in Proposition~\ref{unit} we investigate
the unitality. Proposition~\ref{Odpoledne zavolam.}
establishes the equivalence of the associated categories of operads and
Proposition~\ref{Dnes mam silnou socialni fobii.} addresses algebras.

\vskip .5em
\noindent
{\bf The restriction.}
The ``thin'' operadic category $\R\bO$ is the full subcategory of $\bO$
consisting of objects whose cardinalities are finite ordered sets
$\underline n =
\{1,\ldots,n\} \in \sFin$. Let $\boldf :
S \to T$ be a~morphism in $\bO$, $|S| = \underline m$, $|T| =
\underline n$, and $f : S \to T$ be the same morphism considered as a
morphism in $\R\bO$.

Given $i \in \underline n$, the $i$th 
fiber $f^{-1}(i)$ of $f$ in $\R\bO$ is constructed from the $i$th
fiber $\inv\boldf i$ of $f$ in~$\bO$ as follows. The
cardinality $|\inv\boldf i|$ is a subset of ${\underline m}$. Let $\sigma_i :
|\inv\boldf i| \stackrel\cong\longrightarrow \underline k_i$ be the unique order-preserving isomorphism
to some finite ordinal $\underline k_i \in \sFin$. Then the fiber
$\inv f i$ is defined as the target of the lift $\tilde \sigma_i$ of
$\sigma_i$ in the lifting square
\begin{equation}
\label{Zitra prece jen na letisti budu.}
\xymatrix{\inv\boldf i  \ar@{-->}[r]^{\tilde {\sigma_i}} \ar@{~>}[d]  & \inv
  f i 
\ar@{~>}[d]
\\
|\inv\boldf i |\ar[r]^{\sigma_i}_\cong &\ \underline k_i\ .
}
\end{equation}
Its cardinality is $\underline k_i$.

\vskip .5em
\noindent 
{\bf The extension.}
The objects of the extension $\E\ttO$ are  equivalence
classes of pairs consisting of an isomorphism $\sigma :X \stackrel\cong\longrightarrow
{\underline n}$ and an object $A \in \ttO$ with $|A| = {\underline n}$, written as 
$\rep X\sigma nA$ or as
\begin{equation}
\label{Dnes jsem poslal sparu Ortaggiovi.}
\xymatrix@R=1.3em{&A\ar@{~>}[d]
\\
X \ar[r]^\sigma_\cong  &{\underline n},
}
\end{equation}
modulo the relation
$
\rep X{\ \sigma'} n{A'}
 \sim
\rep X{\ \sigma''} n{A''}
$
if there exists an automorphism  $\phi : {\underline n} \to {\underline n}$
such that $A''$ is the target of the lift $\tilde \phi$ of~$\phi$ in
the lifting square of
\begin{equation}
\label{Dnes pojedeme na chalupu a tam budeme neskutecnych 9 dni.}
\xymatrix@C=3em{
&A' \ar@{-->}[r]^{\tilde \phi}\ar@{~>}[d]  &A'' \ar@{~>}[d]
\\
X   \ar[r]^{{\sigma'}}_\cong &{\underline n}\ar[r]^\phi_\cong &
{\underline n}& X \ar[l]_{{\sigma''}}^\cong
}
\end{equation}
and $\sigma'' = \phi \sigma'$.
The cardinality of such an equivalence class is $X$ by definition.

The morphisms of\, $\E\ttO$ are equivalence classes too; 
a morphism $\boldf : \bS \to
\bT$ from the equivalence class $\bS$ of
$\rep X\xi mS$
to the equivalence class $\bT$ of $\rep Y\eta nT$
is the equivalence class of a~morphism $f:S \to T$ in
\begin{equation}
\label{Dnes jdu s Jarkou na koncert.}
\xymatrix@C=1.6em@R=.7em{
&&S   \ar[rr]^f  \ar@{~>}[dd] && T
\ar@{~>}[dd]
\\
&&&
\\
X\ar[rr]^\xi_\cong &&{\underline m} \ar[rr]^{|f|}  && {\underline n}   &&Y \ar[ll]_\eta^\cong.
}
\end{equation}  
The cardinality functor is given by $|\boldf| := \eta^{-1} |f| \xi : X
\to Y$.
The equivalence, by definition, 
identifies $f:S \to T$ with $\tilde f:= \tilde \sigma
f \tilde \rho^{-1} : \tilde S \to
\tilde T$ in the diagram
\begin{equation}
\label{Pujde s nami dnes Jarka k indianum?}
\xymatrix@C=2.5em@R=1.7em{
&&& \ar[rr]^{\tilde f} \ar@{~>}'[d][dd] \tilde S&& \tilde T \ar@{~>}[dd]
\\
&&S \ar@{-->}[ur]^{\tilde \rho}  \ar[rr]^(.6)f  \ar@{~>}[dd] && 
T \ar@{-->}[ur]^{\tilde
  \sigma}  
\ar@{~>}[dd]
\\
&&&{\underline m} \ar'[r]^(.7){|\tilde f|}[rr] && {\underline n}
\\
X\ar[rr]^\xi_\cong &&{\underline m} \ar[rr]^{|f|}\ar[ur]^\rho_\cong  && {\underline n} \ar[ur]^\sigma_\cong  &&Y \ar[ll]_\eta^\cong 
}
\end{equation}
in which $\rho : {\underline m} \to {\underline m}$ and  $\sigma : {\underline n} \to {\underline n}$ are
isomorphisms in $\sFin$, and $\tilde \rho: S \to \tilde S$,
$\tilde \sigma: T \to \tilde T$ their lifts. 
Notice that 
$\rep X{\rho\xi}m{\tilde  S}$ is another representative of $\bS$ 
and $\rep Y{\sigma\eta}n{\tilde T}$ another representative of
$\bT$, so $\tilde f$ indeed represents a morphism $\bS \to \bT$.

The diagram below defines the categorical composite of a morphism $\bS
\to \bT$ represented by the left part of the diagram, 
followed by a morphism $\bT \to \bR$ represented by the right part,
\[
\xymatrix@C=3em{
&S \ar[r]^f \ar@{~>}[d]   &T' \ar@{~>}[d]  \ar[rr]^{\tilde {\eta ' {\eta''}^{-1}}}&&
T'' \ar[r]^g\ar@{~>}[d]   &R\ar@{~>}[d] &
\\
X\ar[r]^\xi_\cong & \ar[r]^{|f|} \underline m & \underline n 
& \ar[l]_{\eta'}^\cong Y \ar[r]^{\eta'}_\cong   & \underline n \ar[r]^{|g|}
& \underline k &   \ar[l]_\omega^\cong Z.
}
\]
In this diagram, the corner $\rep X\xi mS$ represents $\bS$,
$\rep Y{\eta'} n{T'}$ resp.~ $\rep Y{\eta''} n{T''}$ are two
equivalent representatives of $\bT$, and $\rep Z{\omega} k{R}$
represents $\bR$. The middle square is a lifting square in $\ttO$.

Let us describe the fiber $\inv{\boldf} y$  over $y \in Y$ of a
morphism $\boldf : \bS \to
\bT$  represented by $f: S \to
T$ in~(\ref{Dnes jdu s Jarkou na koncert.}). Denote $i
:= \eta(y)\in {\underline n}$ and let $\inv f i$ be the the $i$th fiber of $f$ in
$\ttO$. In this situation we have the corner
\begin{equation}
\label{Budu se sprchovat asi az zitra.}
\xymatrix@R=3.5em{
&&f^{-1}(i) \ar@{~>}[d]
\\ 
(|f|\xi)^{-1}(i) \ar[r]^\xi_\cong& 
|f|^{-1}(i)\ar[r]^{\it can}_\cong &|f|_{\it pb}^{-1}(i)
}
\end{equation}
where $|f|_{\it pb}^{-1}(i)$ is the $i$th fiber of $|f|$ in $\sFin$,
i.e.\ a pullback, and ${\it can}$ is the unique canonical order-preserving
isomorphism between the set-theoretic preimage $|f|^{-1}(i) \subset {\underline m}$ and the
ordinal   $|f|_{\it pb}^{-1}(i)\in \sFin$. Then  $\inv{\boldf} y  \in
\E\ttO$ is
the equivalence class of the corner~(\ref{Budu se sprchovat asi az zitra.}).

Let us check that our definition of $\inv{\boldf}y$ is
independent of the choice of a representative of $\boldf : \bS \to \bT$.
Consider therefore a representative $\tilde f : \tilde S  \to \tilde
T$ as in~\eqref{Pujde s nami dnes Jarka k indianum?}.
In this situation we have the~diagram
\begin{equation}
\label{Jarka uz dnes podruhe jela.}
\xymatrix@R=10pt@C=20pt{
&&f^{-1}(i) \ar@{~>}[dd]  \ar@{-->}[rr]^{\tilde{\rho_{\it pb}}} && {\tilde f}^{-1}(j) \ar@{~>}[dd]
\\
&\boxed{\mbox{\rm left}} &&\makebox[0pt]{\boxed{\mbox{\rm lifting square}}}&& \boxed{\mbox{\rm right}}
\\ 
(|f|\xi)^{-1}(i) \ar[r]^\xi_\cong& 
|f|^{-1}(i)  \ar[r]^{\it can}_\cong &|f|_{\it pb}^{-1}(i) \ar[rr]^{\rho_{\it pb}} 
&&|\tilde f|_{\it pb}^{-1}(j)&|\tilde f|^{-1}(j)\ar[l]^\cong_{\it can} &
\ar[l]_(.55)\eta^(.55)\cong (|\tilde f| \rho  \xi)^{-1}(j)
}
\end{equation}
in which $j := \sigma^{-1}(i)$, $\rho_{\it pb}$ is the canonical map
between the pullbacks induced by $\rho$,
$\tilde {\rho_{\it pb}}$ its lift, and the meaning of other objects is
similar as in~(\ref{Budu se sprchovat asi az zitra.}). 
Here we \underline{ver}y \underline{crucial}y
needed~(\ref{Zitra pojedeme dale.}) to see that
the target of  $\tilde {\rho_{\it pb}}$ equals the fiber ${\tilde f}^{-1}(j)$.

The left corner in~(\ref{Jarka uz dnes podruhe jela.}) is a representative of
$\inv{\boldf} y$ calculated using $f: S \to T$, the right corner a
representative calculated using $\tilde f: \tilde S \to \tilde T$.
Since clearly $(|f|\xi)^{-1}(i) =  (|\tilde f| \rho  \xi)^{-1}(j)$ and
since the composite 
\[
(|f|\xi)^{-1}(i)   \stackrel \xi\longrightarrow 
|f|^{-1}(i)  \stackrel{{\rm can}}\longrightarrow  |f|_{\it pb}^{-1}(i)
\stackrel{\rho_{\rm pb}}\longrightarrow
 |\tilde f|_{\it pb}^{-1}(j)
\]
equals the composite
\[
(|\tilde f| \rho  \xi)^{-1}(j) \stackrel\eta\longrightarrow 
|\tilde f|^{-1}(j) \stackrel{{\rm can}}\longrightarrow |\tilde f|_{\it pb}^{-1}(j),
\]
the left and the right corners of~(\ref{Jarka uz dnes podruhe jela.}) 
represent the same object of $\E \ttO$,
cf.\ the diagram in~(\ref{Dnes pojedeme na chalupu a tam budeme
  neskutecnych 9 dni.}).
We leave the painstaking verification that the above constructions
produce cloven operadic categories as claimed to the reader.

\begin{proposition}
\label{Vratil jsem se od doktora Reicha.}
There are canonical natural isomorphisms $\R\E
\ttO \cong \ttO$ and \hbox{$\E\R \bO \cong \bO$}  of cloven operadic
categories.
\end{proposition}

\begin{proof}
The pair of mutually inverse functors  
$J :\R\E
\ttO \longleftrightarrow \ttO :I$ is given as follows. The objects of~$\R\E
\ttO$ are the equivalence classes of the corners
\[
\xymatrix@R=1.3em{&A\ar@{~>}[d]
\\
{\underline n} \ar[r]^\sigma_\cong  &{\ \underline n\, }, &\hskip -2em A \in \ttO.
}
\]
The functor $J$ assigns to the equivalence class of the above corner
the domain $\tilde A \in \ttO$ of the lift $\tilde{\sigma}$ in
the lifting diagram
\[
\xymatrix@R=1.3em{
\tilde A \ar@{~>}[d] \ar@{-->}[r]^{\tilde{\sigma}}  & A\ar@{~>}[d]
\\
{\underline n}\ar[r]^{\sigma}_\cong   &{\ \underline n\,}.
}
\]
The inverse $I$ of $J$ sends $T \in \ttO$, $|T|={\underline n}$, 
to the equivalence class of the
corner
\[
\xymatrix@R=1.3em@C=3em{&T\ar@{~>}[d]
\\
{\underline n} \ar[r]^{\rm identity}  &{\ \underline n\,}.
}
\]

Let us describe the pair of mutually inverse functors
$F :\E\R \bO \longleftrightarrow \bO :G$. The objects
of~$\E\R \bO$ are the equivalence classes of corners
\[
\xymatrix@R=1.3em{&\bS\ar@{~>}[d]
\\
X \ar[r]^\varkappa_\cong  &{\underline n}, & \hskip -2em \bS \in \bO.
}
\]
The functor $F$ assigns to the related equivalence class
the domain $\tilde \bS \in \bO$ of the lift $\tilde{\varkappa}$ in
the lifting square
\[
\xymatrix@R=1.3em{
\tilde \bS \ar@{~>}[d] \ar@{-->}[r]^{\tilde{\varkappa}}  & \bS\ar@{~>}[d]
\\
X\ar[r]^{\varkappa}_\cong   &{\ \underline n\,}.
}
\]

To describe the value $G\bS \in \E\R \bO$ of the inverse of $F$ at some
$\bS \in \bO$, $|\bS| = X$, 
choose an arbitrary isomorphism $\omega :
X \to {\underline n}$ and take the lift of $\omega$ as in
\begin{equation}
\label{Poleti se uz zitra?}
\xymatrix@R=1.3em{
\bS \ar@{~>}[d] \ar@{-->}[r]^{\tilde{\omega}}  & \tilde \bS\ar@{~>}[d]
\\
X\ar[r]^{\omega}_\cong   &{\ \underline n\,}.
}
\end{equation}
Then 
$G\bS$ is  the equivalence class of the corner
\[
\xymatrix@R=1.3em{ & \tilde \bS\ar@{~>}[d]
\\
X\ar[r]^{\omega}_\cong   &{\ \underline n\,}.
}
\]
It is simple to check that the above construction leads to
well-defined pairs of mutually inverse functors that preserve the
fibers.
\end{proof}

\begin{corollary}
\label{opet dvanactka}
Each thick cloven operadic category is isomorphic to the extension of
a thin cloven operadic category. Reciprocally, each thin cloven 
operadic category is isomorphic to the restriction of
a thick cloven operadic category.
\end{corollary}

Before we prove, in Proposition~\ref{unit} below, that the
correspondence of \,$\E : \ttO \leftrightarrow
\bO: \R$ preserves the unitality, we formulate

\begin{lemma}
\label{Pujdeme zitra na obed?}
Let\/ $\bO$ be a thick cloven operadic category 
and $\ttO$ the corresponding thin one. Then there is a natural
isomorphism $\pi_0(\ttO) \cong \pi_0(\bO)$ of the sets of connected components. 
\end{lemma}

\begin{proof}
We may assume that $\bO = \E \ttO$, by Corollary~\ref{opet dvanactka},
so we have the inclusion $\ttO \hookrightarrow \bO$ of categories. 
Since each object of $\bO$ is
isomorphic to some object of $\ttO$, cf.~diagram~\eqref{Poleti se uz
  zitra?}, the inclusion induces an epimorphism $\pi_0(\ttO)
\twoheadrightarrow \pi_0(\bO)$. To show that is an isomorphism, we
need to prove that two objects $S$ and $T$, connected by a
zigzag of maps in $\bO$, are connected by a zigzag in $\ttO$.

This  statement follows from the definition of morphisms in $\E
\ttO$, but we prove it directly. Consider e.g.\ the simple zigzag
\begin{equation}
\label{Treti den zavodu a jeste se neletelo.}
\xymatrix{S\ar@{~>}[d]  \ar[r]^{\boldf}\ar@{~>}[d]  
& \ar@{~>}[d] \bR & \ar[l]_\bg  T\ar@{~>}[d] 
\\
{\underline m}  \ar[r]^{|\boldf|}  & X & {\underline n} \ar[l]_{|\bg|}
}
\end{equation}
with $S,T \in \ttO$ and $\bR \in \bO$. Choose an isomorphism $\omega :
X \to \underline k$ and let $R$ be the target of the lift $\tilde \omega$ in
the lifting square
\[
\xymatrix@R=1.3em{
\bR \ar@{~>}[d] \ar@{-->}[r]^{\tilde{\omega}}  & R\ar@{~>}[d]
\\
X\ar[r]^{\omega}_\cong   &\underline k.
}
\] 
Then we replace the zigzag in~\eqref{Treti den zavodu a jeste se
  neletelo.} in $\bO$
by the zigzag $\xymatrix@1{S\ar[r]^f   &R&T \ar[l]_g}$ in $\ttO$, where $f
= \tilde \omega \boldf$ and  $g = \tilde \omega \bg$. More complex
zigzags can be treated similarly. 
\end{proof}

\begin{proposition}
\label{unit}
Let $\ttO$ be a thin cloven operadic category and \,$\bO$ the
corresponding thick one. Then the category $\ttO$ is unital 
if and only if \,$\bO$ is unital.
\end{proposition}

\begin{proof}
By Lemma~\ref{Pujdeme zitra na obed?}, the sets of connected
components of the respective categories are isomorphic. 
Referring to Corollary~\ref{opet dvanactka} we may assume that $\ttO =
\R\bO$. The correspondence between the chosen local terminal objects
in $\ttO$ and the chosen local terminal objects $\bU_c$ in $\bO$ is
described by the lifting square 
\[
\xymatrix@R=1.3em{
U_c \ar@{~>}[d] \ar@{-->}[r]^{\id}  & \bU_c \ar@{~>}[d]
\\
\underline 1\ar[r]^\id   &\{1\}
}
\]
in $\bO$ which actually says that $U_c = \bU_c$. 
\end{proof}

\begin{example}
\label{Stale nevim jako to bude.}
As expected, the category $\sFin$ of finite ordinals and the
category $\Fin$ of finite sets correspond to each other in the
correspondence of Theorem~\ref{Dnes jedu do Kolina a zitra letim do
  Dvora.}. The thick counterpart of the terminal thin unary operadic
category $\tt 1$ is the big groupoid $\mathfrak G$ of one-point sets.   
\end{example}

\vskip .5em
\noindent 
{\bf Operads.}
In  Proposition~\ref{Odpoledne zavolam.} below, $\ttO$ will be a
thin cloven operadic category and $\bO$ the corresponding thick
one, in the correspondence of Theorem~\ref{Dnes jedu do Kolina a zitra letim do
  Dvora.}. We will also assume that the symmetric monoidal category $\ttV$
where our operads live is cocomplete. The notion of categories of
operads is the expected one -- their objects are operads and
their morphisms are morphisms of the underlying collections which commute
with the composition laws, and units in the unital~case.

\begin{proposition}
\label{Odpoledne zavolam.}
The categories of \/ $\ttO$-operads and $\bO$-operads are naturally
equivalent. If \ $\ttO$ or, which is the same, $\bO$ is unital, this
equivalence restricts to the subcategories of unital operads. 
\end{proposition}

\begin{proof}
In view of Corollary~\ref{opet dvanactka} we may assume that $\ttO = \R\bO$.
Given an $\bO$-operad $\boP$, the
underlying collection of the corresponding $\R\bO$-operad $\oP$ 
is the restriction of the underlying collection of $\boP$ to objects of
$\bO$ whose arities are ${\underline n}\in \sFin$, $n \geq 0$. Let us define the
composition laws of $\oP$. 

Consider  a morphism  $\boldf :
S \to T$ of $\bO$, with $|S| = \underline m$, $|T| =
\underline n$,  and denote by $f : S \to T$ the same morphism
interpreted as a morphism in $\R\bO$. Let $\inv {\boldf}i$ be
the fiber of $\boldf$ over $i$ in the category $\bO$, and  $\inv {f}i$ the
$i$th fiber of $f$ in $\ttO$, $i \in \underline n$.
They are connected by the isomorphism 
$\tilde {\sigma_i} :\inv\boldf i  \to \inv f i$ in the lifting
square~\eqref{Zitra prece jen na letisti budu.}. The cloven module
structure of $\boP$ induces the isomorphism
\[
\boP(\sigma) :
\xymatrix@1@C=4em{\oP(f) \cong  \displaystyle\bigotimes_{i\in \underline n} \  \inv f i 
\ar[r]^{\bigotimes_{i\in \underline n} \boP({\sigma_i})}
  & \ \displaystyle \bigotimes_{i\in \underline n}
  \inv{\boldf} i \cong \boP(\boldf).
}
\]
The composition law $\mu_f: \oP(T) \ot \oP(f) \longrightarrow \oP(S)$ of $\oP$ 
is then defined by the diagram
\[
\xymatrix@C=5em{
\oP(T) \ot \oP(f)   \ar[d]_{\id\ \ot \boP(\sigma)}   
\ar[r]^{\mu_f}& \oP(S) \ar@{=}[d]
\\
\boP(T) \ot \boP(\boldf) \ar[r]^{\bmu_\boldf}& \ \boP(S),
}
\]
where $\bmu_\boldf$ in the bottom line is the composition law of $\boP$.

Let us explain how an $\R\bO$-operad $\oP$ determines the corresponding $\bO$-operad
$\boP$. Each \hbox{$X \in \Fin$} determines the
groupoid ${\tt G}(X)$, whose 
objects are isomorphisms $\sigma: X \stackrel\cong\to \underline n$ with some
(necessarily unique) 
$\underline n$, and morphisms $\Omega :\sigma' \to \sigma''$ are 
commutative diagrams
\begin{equation}
\label{Musim zacit uhanet grantisty a ty dva lemply.}
\xymatrix@R=1.2em@C=1.2em{& \ar[dl]_{\sigma'} \ar[dr]^{\sigma''}   X&
\\
\underline n \ar[rr]^\omega && \underline n.
}
\end{equation}
of isomorphisms in $\Fin$.
For each $\bS\in \bO$ with $|\bS| = X$, 
$\oP$ defines a contravariant functor  \hbox{$\oP_\bS:{\tt G}(X) \to \ttV$}
as follows. 
Its value $\oP(\sigma)$ on $\sigma :X \stackrel\cong\to \underline n \in {\tt G}(X)$ is $\oP(S_\sigma)$, where
$S_\sigma$ is the right corner of the lifting square
\begin{equation}
\label{Dneska si Jarka vezme tac.}
\xymatrix@R=1.3em{\bS \ar@{-->}[r]^{\tilde \sigma}\ar@{~>}[d] & S_\sigma\ar@{~>}[d]
\\
X\ar[r]^{\sigma}_\cong   &{\underline n}
}
\end{equation}
in $\bO$.
The value  $\oP_\bS(\Omega)$ on the
morphism $\Omega$ in~\eqref{Musim zacit uhanet grantisty a ty dva
  lemply.} is given as follows. 
The commutative diagram in~\eqref{Musim zacit uhanet grantisty a ty dva
  lemply.} determines the diagram of three lifting squares
\[
\xymatrix@C=4em{
\ar@/^2em/@{-->}[rr]^{\tilde{\omega}}
S_{\sigma'} \ar@{<--}[r]^{\tilde {\sigma'}} \ar@{~>}[d]& \bS \ar@{-->}[r]^{\tilde {\sigma''}}  \ar@{~>}[d]&S_{\sigma''} \ar@{~>}[d]
\\
\underline n \ar@{<-}[r]^{\sigma'}_\cong 
\ar@/_2em/[rr]^{{\omega}}_\cong
&X  \ar[r]^{\sigma''}_\cong&
\underline n  
}
\]  
with commutative upper and lower bases. The cloven module action on $\oP$ induces
the morphism $\oP(\omega) : \oP(S_{\sigma''}) \to  \oP(S_{\sigma'})$
which we interpret as $\oP_\bS(\Omega): \oP(\sigma'') \to
\oP(\sigma')$. This finishes the definition of the functor $\oP_\bS$.
The piece $\boP(\bS)$ of the 
underlying collection of the $\bO$-operad $\boP$ is the colimit
\begin{equation}
\label{Pomuze si tim nejak?}
\boP(\bS) := \colim  \oP_{\!\bS}, \ \bS \in \bO.
\end{equation}
Since the indexing category ${\tt G}(X)$ is a groupoid, each
isomorphism $\sigma :
X \stackrel\cong\to \underline n$ determines a natural isomorphism
\begin{equation}
\label{Jarka dnes zustava doma.}
\pi_\sigma : \oP(S_\sigma) \stackrel\cong\longrightarrow \boP(\bS).
\end{equation}

Assume that $\boldf : \bS \to \bT$ is a morphism in $\bO$, $|\bS| = X$,
$|\bT| = Y$. To define
the associated composition law $\bmu_\boldf : \boP(\bT) \ot
\boP(\boldf) \to \boP(\bS)$, choose two 
isomorphisms $\rho : X \stackrel\cong\to \underline
m$, $\sigma : Y \stackrel\cong\to \underline
n$, and construct the diagram
\begin{equation}
\label{Prvni den piji 3 litry.}
\xymatrix@C=2.3em@R=1.5em{& 
  \ar@{-->}[rr]^{f_{\rho\sigma}} \ar@{~>}'[d][dd] S_\rho
&& T_\sigma \ar@{~>}[dd]
\\
\bS \ar@{-->}[ur]^{\tilde \rho}  \ar[rr]^(.6)\boldf  \ar@{~>}[dd] && 
\bT \ar@{-->}[ur]^{\tilde
  \sigma}  
\ar@{~>}[dd]
\\
&\underline m\ar@{-->}'[r]^(.7){|f_{\rho\sigma}|}[rr] && \underline  n
\\
X \ar[rr]^{|\boldf|}\ar[ur]^\rho_\cong  && Y \ar[ur]^\sigma_\cong
}
\end{equation}
in which $f_{\rho\sigma} := {\tilde \sigma} \boldf {\tilde \rho}^{-1}$
is a morphism in $\ttO = \R \bO$ and the left and right faces are
lifting squares.
Given $y \in Y$ and $i := \sigma(y)
\in \underline n$, the diagram~\eqref{7my den na chalupe} associated
to~(\ref{Prvni den piji 3 litry.}) 
leads to the left lifting square~of
\begin{equation*}
\xymatrix{
\inv {\boldf}y \ar@{~>}[d]  \ar[r]^{\tilde {\rho_y}}  & 
 \ar[r]^{\tilde {\rho_i}} \ar@{~>}[d] \boldf_{\rho\sigma}^{-1}(i)
 &\ar@{~>}[d] 
\inv {f_{\rho\sigma}}i 
\\
\inv{|\boldf|}y \ar[r]^{\rho_y}  &  \ar[r]^{\rho_i}
\inv{|\boldf_{\rho\sigma}|}i 
&\ \inv{|f_{\rho\sigma}|}i \,.
}
\end{equation*}
The right lifting square expresses the relation between the fiber
$\boldf_{\rho\sigma}^{-1}(i)$ of $f_{\rho\sigma}$ considered as 
a~morphism in $\bO$, and the fiber $\inv {f_{\rho\sigma}}i$ 
of $f_{\rho\sigma}$ in $\ttO =
\R\bO$, cf.~\eqref{Zitra prece jen na letisti budu.}.
The isomorphisms 
\[
\pi_{\rho_i\sigma_y}:  \oP(\inv
{f_{\rho\sigma}}i)  \stackrel\cong\longrightarrow
\boP(\inv{\boldf}y),
\ y \in Y,\ i := \sigma(y),
\]
assemble to an isomorphism 
$\pi_{\rho\sigma}: \oP({f_{\rho\sigma}}) \cong \boP(\boldf)$. 
We finally define the composition
law $\bmu_\boldf$ of $\boP$ via
the diagram 
\[
\xymatrix@C=3em{
\boP(\bT) \ot \boP(\boldf) \ar[r]^(.6){\bmu_{\boldf}}   & \boP(\bS)
\\
\oP(T_\sigma) \ot \oP(f_{\rho\sigma})
\ar[r]^(.6){\mu_{f_{\rho\sigma}}}
\ar^{\pi_\sigma \ot \pi_{\rho\sigma}}_\cong[u]& \oP(S_\rho)\ar_{\pi_\rho}^\cong[u]
}
\]
in which  $\mu_{f_{\rho\sigma}}$ is composition law of the 
$\ttO$-operad $\oP$. It can be verified that this definition does not depend on the
choices of the isomorphisms $\rho$ and $\sigma$ in~\eqref{Prvni den piji 3 litry.}. 

It remains to discuss the unitality, assuming as before that $\ttO = \R\bO$. Let
us invoke Lemma~\ref{Pujdeme zitra na obed?} and take $c \in
\pi_0(\ttO) \cong \pi_0(\bO)$.  
The chosen local terminal object $U_c$ of~$\ttO$ is related to the
corresponding chosen local terminal object of $\bO$ by the trivial lifting
square
\[
\xymatrix@R=1.3em{\bU_c \ar@{-->}[r]^{\id}\ar@{~>}[d] & U_c\ar@{~>}[d]
\\
\{1\}\ar[r]^\id   &{\underline 1}
}
\]
If an $\ttO$-operad  $\oP$ is the restriction of an $\bO$-operad
$\boP$ to $\ttO$, then $\oP(U_c)
= \boP(\bU_c)$ since $\bU_c \in \ttO =
\R\bO$. We define the unit maps for $\oP$ using the unit
maps $\eta_c : \unit \to \boP(\bU_c)$ of $\boP$ as the composite
\[
\unit\stackrel{\eta_c}\longrightarrow  \boP(\bU_c) \stackrel\id\longrightarrow  \oP(U_c).
\]
Let, on the other hand, $\boP$ be the extension of $\oP$. Then the
unit maps of  $\boP$ are defined using the unit maps  $\eta_c : \unit
\to \oP(U_c)$ of $\oP$ as the composite
\[
\unit \stackrel{\eta_c}\longrightarrow \oP(U_c)
\stackrel{\pi_\id}\longrightarrow \boP(\bU_c)
\]
which uses the canonical isomorphism~\eqref{Jarka dnes zustava doma.}.
We graciously leave to the reader the arduous task of verifying that
the above constructions do what we claim they do.
\end{proof}

\begin{example}
The necessity of the compatibility~\eqref{Vcera jsem byl se Sambou v
  Krkonosich.} of the composition law~\eqref{structure operations} 
with the cloven action for Proposition~\ref{Odpoledne zavolam.} to
hold can be seen in the following simple example. 
Let $\tt 1$ be the terminal unary thin operadic category,
cf.~non-Example~\ref{tyden po vysetreni}, and
$\mathfrak G$  its thick counterpart introduced 
in~Example~\ref{Stale nevim jako to bude.}.
While the $\tt 1$-operads are associative algebras, the $\mathfrak
G$-operads without~\eqref{Vcera jsem byl se Sambou v
  Krkonosich.} are families $\{\boP(x) \,|\, \{x\}
\in \mathfrak G\}$ indexed by singletons, with left actions
\[
\vartriangleright \, : \boP(x) \ot \boP(y) \longrightarrow \boP(y)
\]
such that the diagram
\[
\xymatrix{\boP(x) \ot \boP(y) \ot \boP(z)
  \ar[d]_{\vartriangleright \ot \id}
\ar[r]^(.57){\id\, \ot \vartriangleright}
& \boP(x) \ot \boP(z) \ar[d]^\vartriangleright
\\
\boP(y) \ot \boP(z) \ar[r]^\vartriangleright   & \boP(z)
}
\] 
commutes for each singletons $\{x\}, \{y\},\{z\} \in \mathfrak G$.
  \end{example}

\vskip .5em
\noindent 
{\bf Algebras.}
In the following proposition, $\ttO$ will be a thin cloven operadic
category, $\bO$ the corresponding thick one, $\oP$ an $\ttO$-operad and
$\boP$ the corresponding $\bO$-operad. The notion of categories of
algebras is again the expected one.

\begin{proposition}
\label{Dnes mam silnou socialni fobii.}
The categories of \/ $\boP$- and \, $\oP$-algebras are naturally
equivalent. If we declare collections with canonically isomorphic
indexing sets to be {\/\em the same\/}, then this equivalence is
actually an isomorphism of categories.  
\end{proposition}

\begin{proof}
By definition, the underlying collection of a $\oP$-algebra $A$ is of the
form $\{A_c\}_{c
\in \pi_0(\ttO)}$, while the underlying collection of a $\boP$-algebra $\bA$
is of the form $\{\bA_c\}_{c \in \pi_0(\bO)}$. In view of
Lemma~\ref{Pujdeme zitra na obed?} which we use to identify
$\pi_0(\ttO)$ with $\pi_0(\bO)$, the underlying  collections of
$\oP$-algebras and $\boP$-algebras are therefore of the same
form. The equivalence of the proposition acts as the identity on the
underlying collections.

Replacing the operadic category $\ttO$ by an isomorphic
one if necessary, we may assume, as in the proof of Proposition~\ref{Odpoledne
  zavolam.}, that $\ttO = \R \bO$, and refer to the explicit correspondence
between \hbox{$\bO$-operads} and $\ttO$-operads constructed there.

If $\bA = \{\bA_c\}_{c \in \pi_0(\bO)}$ is a $\boP$-algebra, the  
$\oP$-action on the corresponding $\oP$-algebra $A= 
\{A_c\}_{c \in \pi_0(\ttO)}$ is the
restriction of the $\boP$-action to the objects 
of~$\R \bO$. On the other hand, given a $\oP$-algebra $A$, 
we have, in the notation of the second part of the proof of
Proposition~\ref{Odpoledne zavolam.}, cf. 
diagrams~\eqref{Musim zacit uhanet grantisty a ty
  dva lemply.}--\eqref{Jarka dnes zustava doma.} in particular, 
the diagram
\[
\xymatrix{
\oP(S_{\sigma''}) \ot  A_{s(S_{\sigma''})} 
\ar[r]^(.66){\alpha_{S_{\sigma''}}} 
\ar[d]_{\oP(\omega) \otimes A_{s(\omega)}} & \ar@{=}[d]
  A_{t(S_{\sigma''})} 
\\
\ar[r]^(.66){\alpha_{S_{\sigma'}}}
\oP(S_{{\sigma'}}) \ot  A_{s(S_{\sigma'})}  &
  A_{t(S_{{\sigma'}})} .
}
\] 
which is commutative 
by item (ii) of Definition~\ref{Jaruska uvarila brokolici.}. Passing to the
colimit~\eqref{Pomuze si tim nejak?} gives the action 
\[
\bal_\bS :\boP(\bS) \ot A_{\bs(\bS)} \to A_{\bs(\bT)}.
\]
It is simple to verify that the above constructions preserve unitality.
\end{proof}

\section{Graphs, modular operads, \&c}
\label{co tri dny}

The aim of this section is to introduce a thick version of the operadic
category of graphs and show how they describe modular, resp.~odd
modular operads in the form of~\cite[Definitions~A.1,~A.2]{kodu} when arities 
are finite sets rather than natural
numbers. The presentation of the category of graphs 
here follows closely~\cite[Section~3]{env}. We will
ignore the genus grading of (odd) modular operads, as it does not 
bring anything conceptually new to the picture.

\begin{definition} 
\label{pre}
A {\em thick graph\/} $\Gamma$ is a pair $(g,\sigma)$ consisting
of an  order-preserving map 
\[
g:F\to V, \ V\ne \emptyset,
\] 
in the category $\Fin$ of finite sets together with an
involution $\sigma$ on $F$.
\end{definition}

Elements of $\Flag(\Gamma) := F$ are the {\em flags\/} 
(also called {\em half-edges\/})
of $\Gamma$ and elements of $\Vert(\Gamma) := V$ are its
{\em vertices\/}. The elements of the set $\Leg(\Gamma)\subset F$ of fixed
points of $\sigma$ are called the {\em legs\/} (also called 
{\/\em hairs\/}\/) of
$\Gamma$ while nontrivial orbits $\Edg(\Gamma)$ of $\sigma$ are its {\em edges\/}. The
{\em endpoints\/} 
of an edge $e = \{h_1,h_2\}\in \Edg(\Gamma)$ are $g(h_1)$ and $g(h_2)$.
We will use the notation $(F,V)$ or $(F,g,V)$ for a
graph $\Gamma =(g,\sigma)$ if we want to specify its set of
vertices and flags. 

A thin version of graphs in Definition~\ref{pre} was called in~\cite{env} 
{\em preordered\/} graphs, where ``preordered'' indicated
the absence of the global orders of the legs. In our situation,
nothing is ordered and this terminology would be misleading.

A {\em morphism\/} of graphs $\Phi:\Gamma'\to \Gamma''$ is a pair $(\psi,\phi)$
of morphisms of finite sets such that the diagram
\begin{equation}
\label{graphsmorphism}
\xymatrix@C=3.5em{F' \ar[d]_{g'}  & F''\ar@{_{(}->}[l]_{\psi}
\ar[d]^{g''}
\\
V' \ar@{->>}[r]^{\phi} &V''
}
\end{equation}
commutes. We moreover require $\phi$ to be a surjection, and $\psi$ 
equivariant with respect to the involutions, and bijection on fixed
points. Thus $\psi$
injectively maps flags to flags and bijectively legs to legs.
The pair $(\psi,\phi)$ must satisfy the following condition: 
If $\phi(x) \ne \phi(y)$, $x,y \in V'$ and
$e'$ is an edge of $\Gamma'$ with endpoints $x$ and $y$, then there exists an edge
$e''$ in $\Gamma''$ with endpoints $\phi(x)$ and $\phi(y)$ such that $e'=
\psi(e'')$. Graphs and their morphisms form a category $\prGR$.

The {\em fiber\/} $\inv{\Phi}x$ 
of a map $\Phi = (\psi,\phi): \Gamma' \to \Gamma''$ 
in~(\ref{graphsmorphism}) over $x\in V''$ is a
graph whose set of vertices is $\phi^{-1}(x)$ and whose
set of flags is $(\phi g)^{-1}(x)$.
The involution $\tau$ of $\inv{\Phi}x$ 
is defined by 
\[
\tau(h) := 
\begin{cases}   h& \hbox {if $h\in \Im(\psi)$,}
\\ 
\sigma(h)& \hbox {if $h\notin \Im(\psi)$,} 
\end{cases}
\]
where $\sigma$ is the involution of $\Gamma$.

For a finite set $X \in \Fin$, denote by $C_X$ the graph $X \to
\{1\}$ with the trivial involution on~$X$. It~can be viewed as a corolla
with spikes indexed by $X$ and the hub labelled by $\{1\}$. 
Let $\GR$ be the coproduct of
the categories ${\prGR}/C_X$ for $X \in \Fin$. 
We are aware that we are skating on thin ice here, but we may always
assume that we are working in the universe which is sufficiently big
so that the coproduct makes sense. Moreover, we will always know  in
concrete situations what
we are doing. Explicitly, an object
of  ${\prGR}/C_X$ is a morphism in $\prGR$ of the form
\begin{equation}
\label{V utery MR.}
\xymatrix@C=3.5em{F \ar[d]_{g}  & X\ar@{_{(}->}[l]_{\psi}
\ar[d]^{!}
\\
V \ar@{->>}[r]^{!} &\ \{1\}.
}
\end{equation}
The inclusion $\psi :X \hookrightarrow F$ induces an isomorphism between
$X$ and the set of legs of $\Gamma = (V,g,F)$. The set
$X$ thus appears as the set of labels of the legs of $\Gamma$, which
we call the {\/\em global\/} labels of $\Gamma$. The morphisms of 
$\GR$ are the morphisms of $\prGR$ that keeps the global labels fixed.

\begin{proposition}
\label{Jeste mi nevolali.}
The category \/ $\GR$ is a thick operadic category. 
\end{proposition}

\begin{proof}
The cardinality functor assigns to a
graph  $\Gamma$ its set $\Vert(\Gamma)$ of vertices. 
We describe the fiber structure  and
leave the verification of the axioms as an exercise. 
A morphisms $\Phi : \Gamma' \to \Gamma''$
in $\GR$ is, by definition, a diagram
\begin{subequations}
\begin{equation}
\label{Ruda_je_rasista.}
\xymatrix@R=.3em@C=2.5em{
F'  \ar[dddd]_{g'} && \ar@{_{(}->}[ll]_\psi \ar[dddd]^{g''} F''
\\
&X \ar@{^{(}->}[lu] \ar[dd]  \ar@{_{(}->}[ru] &
\\
{}
\\
&\{1\}  &
\\
V' \ar@{->>}[ru]^{!} \ar[rr]^\phi &&\ V''.\ar@{->>}[lu]_{!}
}
\end{equation}
It induces for each $x \in V''$ a commutative diagram
\begin{equation}
\label{Cekam na vysedky a bojim se.}
\xymatrix@C=3.5em{(\phi g')^{-1}(x) \ar[d]_{g'_x}  
& 
(g'')^{-1}(x)\ar@{_{(}->}[l]_(.5){\psi_x}  \ar[d]^{g''_x}
\\
\phi^{-1}(x) \ar[r]^{\phi_x} & \{1\}
}
\end{equation}
\end{subequations}
of finite sets in which
the morphisms $g'_x,\phi_x,g''_x$ and $\psi_x$ are the restrictions of the 
corresponding morphisms from~(\ref{Ruda_je_rasista.}). 

We interpret the right vertical morphism of~\eqref{Cekam na vysedky a
  bojim se.} as a corolla $C_X$ with $X :=
(g'')^{-1}(x)$, imposing the
trivial involution on $X$. Due to the definition of
fibers of maps in $\prGR$, the diagram above represents a map 
of the fiber $\inv\Phi x$ in $\prGR$ to $C_X$, which makes it an
object of~$\GR$. We take it as the definition of the fiber {\em in \/}  $\GR$. 
So the fiber gets its global labels from the set
$(g'')^{-1}(x)$.
\end{proof}

\begin{proposition}
The category \, $\GR$ is a thick cloven unital category. 
\end{proposition}

\begin{proof}
Let us start with the cloven structure. Given a
graph  $\Gamma = (V,g,F) \in \GR$ with the
global labels $X$ and an isomorphism  $\phi: V \to W$  
of finite sets, we define the graph $\tilde \Gamma$ and an
isomorphism $\tilde \phi$ in the lifting square
\[
\xymatrix@R=1.3em{
\Gamma \ar@{-->}[r]^{\tilde \phi}\ar@{~>}[d] & \tilde \Gamma \ar@{~>}[d]
\\
V\ar[r]^{\phi}_\cong   &W
}
\]
as follows. We put $\tilde \Gamma = (W,\tilde g ,F)$, with the
involution $\sigma: F \to F$ taken from $\Gamma$ and $\tilde g :=
\phi g$. The morphism $\tilde \phi$ is given by the diagram
\[
\xymatrix@R=.3em@C=2.5em{
F  \ar[dddd]_{g} && \ar@{_{(}->}[ll]_{\id_F} \ar[dddd]^{\tilde g} F
\\
&X \ar@{^{(}->}[lu] \ar[dd]  \ar@{_{(}->}[ru] &
\\
{}
\\
&\{1\}  &
\\
V \ar@{->>}[ru] \ar[rr]^\phi &&\ W.\ar@{->>}[lu]
}
\]

The (big) set $\pi_0(\GR)$ of connected components is $\Fin$. 
A graph $\Gamma \in \GR$ belongs to the component $X \in \Fin$ if and only if its set
of global labels equals $X$. The chosen local terminal
objects~\eqref{Pisu v Mercine} 
are the corollas $\bigstar_X : C_X \stackrel\id\to C_X$. 
In words, the chosen terminal objects are corollas such that the
global labels coincide with the (labels of) the spikes.
Figure~\ref{Popea} shows
the local terminal object $\bigstar_X$ with $X = \{a,b,c,d\}$. The labels of the legs
are encircled.
\begin{figure}
\begin{center}
\includegraphics{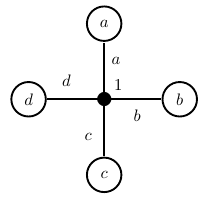}
\end{center}
\caption{The local terminal object $\bigstar_X$ of $\GR$ with $X = \{a,b,c,d\}$.\label{Popea}}
\end{figure}
We leave again the details to the reader.
\end{proof}

The thin operadic category corresponding to $\GR$, realized as the restriction $\gr
:=\R\GR$ in the proof of Theorem~\ref{Dnes jedu do Kolina a
  zitra letim do Dvora.}, is a simple modification of the constructions
above. The main difference is that the sets of vertices of graphs in $\gr$ are finite
ordinals $\underline n$, $n \geq 1$. All definitions translate
verbatim, only the diagram~\eqref{Cekam na vysedky a bojim se.} has
now the form
\[
\xymatrix@C=3.5em{(\phi g')^{-1}(i) \ar[d]_{g'_i}  
& 
(g'')^{-1}(i)\ar@{_{(}->}[l]_(.5){\psi_i}  \ar[d]^{g''_i}
\\
\phi_{\it pb}^{-1}(i) \ar[r]^{\phi_i} & \underline 1
}
\]    
with $i \in \underline n$, where $g'_i$ is the composite of the
restriction $(\phi g')^{-1}(i) \to \phi^{-1}(i)$ of $g'$ with the
unique order-preserving isomorphism of the finite ordered set
$\phi^{-1}(i)$ with the unique finite ordinal $\phi_{\it
  pb}^{-1}(i)$. 

As stated in~\cite[Proposition~5.11]{kodu}, algebras for the constant operad over the
category $\Gr$ are modular operads with the arities in finite
ordinals. Let us see what happens in our setup.
Let thus~$\bzeta$ be the constant unital $\GR$-operad given by $\bzeta(\Gamma):=
\unit$ for each $\Gamma \in \GR$ and the composition~law
\[
\mu_\Phi : \bzeta(\Gamma'')\ \ot \ \bzeta(\Phi) = \unit \ot \unit 
\longrightarrow \unit = \bzeta(\Gamma')
\]
for $\Phi : \Gamma' \to \Gamma''$ a morphism in $\GR$.

It will not be difficult to investigate $\bzeta$-operads directly, but
we can
further simplify our life by looking at its restriction  $\zeta$ to
$\gr$. Both operads $\bzeta$
and $\zeta$ have equivalent, ``almost'' isomorphic, categories of algebras 
by Proposition~\ref{Odpoledne zavolam.}. An algebra for a
$\gr$-operad is, by definition, a collection indexed by $\pi_0(\gr) =
\Fin$. A $\zeta$-algebra thus sits on a collection $\M = \{M(X) \ | \
X \in \Fin\,\}$. 

Let us investigate the structure operations, starting with the
one-vertex graph, the corolla 
$\bigstar_\omega \in \gr$ given, in the representation used 
in~\eqref{V utery MR.}, 
by the diagram  
\[
\xymatrix@C=3.5em{X' \ar[d]_{!}  & X''\ar[l]_{\omega}^\cong
\ar[d]^{!}
\\
\underline 1 \ar@{=}[r] &\ \underline 1,
}
\]
in which $\omega$ is an isomorphism of finite sets and the involutions
on $X'$ and $X''$ are trivial. Its list of
sources contains only  $X'$, its target is $X''$. The associated operations
\begin{equation}
\label{1}
\alpha_{\bigstar_\omega}:
\M(X') \cong \zeta(X') \ot \M(X') \longrightarrow \M(X'') 
\end{equation}
make $\M$ a module over the groupoid of finite sets and their isomorphism.
Let $X$, $Y$ be disjoint finite sets and $x \not\in X$, $y \not\in
Y$, $x \not= y$. Consider the two-vertex graph given by
\begin{equation}
\label{Dnes dela Jaruska jablicka v zupanu.}
\xymatrix@C=3.5em{X \cup \{x,y\} \cup Y \ar[d]_{g}  & X \cup 
Y\ar@{_{(}->}[l]_(.4){\psi}
\ar[d]^{!}
\\
\underline 2 \ar@{->>}[r]^{!} &\ \underline 1,
}
\end{equation}
with the involution $\sigma$ acting trivially on $X \cup Y$,
$\sigma(x) = y$, $\psi$ the inclusion and 
the vertex map $g$ which equals $1$ on $X \cup \{x\}$ and $2$ on $Y \cup \{y\}$.
It has the sources $X \cup \{x\},
Y\cup \{y\}$, and the target $X \cup Y$. The associated operation
is a map
\begin{equation}
\label{2}
{}_x\circ_y : \M(X \cup \{x\}) \ot \M(Y\cup \{y\}) \longrightarrow \M(X \cup Y).
\end{equation}
Consider finally the one vertex graph, the tadpole
\begin{equation}
\label{Prinesu si slehacku.}
\xymatrix@C=3.5em{X \cup \{a,b\} \ar[d]_{!}  & X \ar@{_{(}->}[l]_(.4){\psi}
\ar[d]^{!}
\\
\underline 1 \ar@{=}[r] &\ \underline 1,
}
\end{equation}
with the involution
$\sigma$ acting trivially on $X$, $\sigma(a) = b$, and $\psi$ the
inclusion. It has only one source $X \cup \{a,b\}$ and the target $X$. 
The corresponding operation has the form
\begin{equation}
\label{3}
\circ_{ab} : \M(X \cup \{a,b\}) \longrightarrow \M(X).
\end{equation}

\begin{proposition}
\label{Je utery a uz jsem docela unaven.}
The operations~\eqref{1},~\eqref{2} and~\eqref{3} induce on $\M$ the structure of a
modular operad in the presentation of\/~\cite[Definition~A.1]{kodu}.
\end{proposition}

\begin{proof}
The proof is practically the same as that
of~\cite[Proposition~5.11]{kodu}. We are not going to repeat it
here. The graph in~\eqref{Dnes dela Jaruska
  jablicka v zupanu.} is an analog of the graph in~(69b)
of~\cite{kodu}, and the graph in~\eqref{Prinesu si slehacku.} an
analog of the graph in~(69a) loc.~cit.
\end{proof}

Suppose for a moment that the base monoidal category $\ttV$ is the
category of graded vector spaces. 
For a vector space $A$ of dimension $k$, we denote by $\det(A) :=
\hbox {\large$\land$}^k(A)$ the
top-dimensional piece of its Grassmann algebra. If $S$ is a nonempty finite set,
we let $\det(S)$ to be the determinant of the vector space spanned
by $S$. Let $\bzetaodd$ ($\bzeta$ turned upside down) be the
$\GR$-operad defined by  $\bzetaodd(\Gamma) := \det(\Edg(\Gamma))$,
the determinant of the internal edges of $\Gamma$, with the
composition law given by the natural isomorphism
\begin{equation}
\label{pred 9 dny}
\mu_\Phi:\bzetaodd(\Gamma'') \ot \bzetaodd(\Phi) =
\det(\Edg(\Gamma'') \ot \bigotimes_{v \in \Vert(\Gamma'')}
\Edg(\Gamma_v) 
\stackrel\cong\longrightarrow
\det(\Edg(\Gamma') = \bzetaodd(\Gamma'),
\end{equation} 
where  $\Gamma_v$, $v \in \Vert(\Gamma'')$ are the fibers of a
map $\Phi : \Gamma' \to \Gamma''$. The isomorphism uses the fact
that the set of edges of $\Gamma'$ is the union of the set
of edges of\/ $\Gamma''$ with the sets of edges of the fibers of
$\Phi$. In~(\ref{pred 9 dny}) we thus see an isomorphism of
one-dimensional vector spaces given by multiplication with the sign
resulting from the reordering the factors. We formulate

\begin{proposition}
The algebras for the constant\/ $\GR$-operad $\bzeta$ are modular
operads in the presentation~\cite[Definition~A.1]{kodu}. 
The algebras for the odd\/ $\GR$-operad $\bzetaodd$ are odd modular
operads in the presentation~\cite[Definition~A.2]{kodu}. 
\end{proposition}

\begin{proof}
The first part is a reformulation of Proposition~\ref{Je utery a uz
  jsem docela unaven.}. The second part can be established in a
similar way, taking into account the nontrivial signs coming 
from the isomorphism in~\eqref{pred 9 dny}.
\end{proof}

In the same way we can get the thick versions of the operadic
categories adapted for classical operads, wheeled PROPs, dioperads and
$\frac12$PROPs respectively, listed in diagram (37) of~\cite{env}.

\appendix

\section{Semi-ordered operadic categories}

It turns out that {\em every\/} -- no cleavage needed -- standard
(thin) operadic category has a thick counterpart with cardinalities in
{\em ordered} finite sets. But this fact is not very interesting
since it does not remove orders from the picture. We include 
the related material just to finish our story.
Let us denote by $\semi$ the category of ordered sets and their {\em all\/},
not necessarily order-preserving, maps. 

\begin{definition}
A thick operadic category $\bO$ is {\/\em semi-ordered\/} if a
factorization of the
cardinality functor $\crd: \bO \to \Fin$ through the
category $\semi$ is specified. A semi-ordered operadic category is {\/\em cloven\/} if
functorial lifts  $\tilde \sigma$  in~\eqref{Pujdeme k Pakousum?}
of order-preserving isomorphisms $\sigma : X\to Y$ are specified.
\end{definition}

Since every finite set can be ordered, under a mild categorial
assumptions, every functor to finite sets factorizes through $\semi$,
every operadic category admits a semi-ordered structure. However,
the chosen factorization determines the class of isomorphisms required to
have functorial lifts, so not every semi-ordered operadic category is cloven. 
The notion of a cloven $\bO$-module is easily translated to the cloven
semi-ordered case; we require the
action of $\tilde \sigma$ in~\eqref{Nemohu najit tu
  propisovacku.} only if $\sigma$ is order-preserving. 
The definitions of operads and their algebras remain the same.

Notice that thin operadic categories are always ``cloven semi-ordered,''
since the image of the cardinality functor already consists of
ordered sets, and the only order-preserving isomorphism $\sigma :
\underline n \to  \underline n$ that must admit functorial lift 
is the identity automorphism.
A semi-ordered operadic category (thin or thick) is {\/\em ordered\/},
if the image of the chosen factorization lies within the category of ordered
sets and their order-preserving~maps.

\begin{theorem}
\label{V patek exkurze do Caslavi.}
There is a natural correspondence \ $\E : \ttO \leftrightarrow
\bO: \R$ (extension and restriction) which extends to an equivalence
of the categories of (standard, thin) operadic categories with the
category of semi-ordered cloven thick operadic categories.
The corresponding operadic
categories have equivalent categories of operads and their
algebras.
This equivalence restricts to the subcategories of ordered operadic categories.
\end{theorem}

\begin{proof}
A very simplified version of the proof of
Theorem~\ref{Dnes jedu do Kolina a zitra letim do Dvora.}. The
simplification comes from the fact that the construction of the
extension $\E$ does not use equivalence classes, since the objects of $\E
\ttO$ are diagrams~\eqref{Dnes jsem poslal sparu Ortaggiovi.} with 
$\sigma$ the unique order-preserving isomorphism.
\end{proof}



\end{document}